\documentclass[12pt]{article}

\usepackage{tikz}
\usepackage{nameref}
\usepackage{empheq}
\usepackage{comment}
\usepackage[shortlabels,inline]{enumitem}
\setlist[enumerate]{nosep}
\usepackage[colorlinks=true,
linkcolor=refkey,
urlcolor=lblue,
citecolor=red]{hyperref}
\usepackage[doc,wmm,hhb]{optional}
\usepackage{xcolor}

\usepackage{float}
\usepackage{soul}
\usepackage{graphicx}
\definecolor{labelkey}{rgb}{0,0.08,0.45}
\definecolor{refkey}{rgb}{0,0.6,0.0}
\definecolor{Brown}{rgb}{0.45,0.0,0.05}
\definecolor{lime}{rgb}{0.00,0.8,0.0}
\definecolor{lblue}{rgb}{0.5,0.5,0.99}
\definecolor{OliveGreen}{rgb}{0,0.6,0}
\definecolor{tyrianpurple}{rgb}{0.4, 0.01, 0.24}
\usepackage{mathpazo}

\colorlet{hlcyan}{cyan!30}

\usepackage{stmaryrd}
\usepackage{amssymb}
\makeatletter
\def\namedlabel#1#2{\begingroup
	\def\@currentlabel{#2}%
	\label{#1}\endgroup
}
\makeatother

\newcommand{\seppthree}{\setlength{\itemsep}{-3pt}}

\newcommand{\SE}{\ensuremath{{\mathcal S}}}

\newcommand{\J}[1]{\ensuremath{{\operatorname{J}}_{#1}}}
\newcommand{\R}[1]{\ensuremath{{\operatorname{R}}_%
		{#1}}}
\newcommand{\Pj}[1]{\ensuremath{{\operatorname{P}}_%
		{#1}}}

\newcommand{\N}[1]{\ensuremath{{\operatorname{N}}_%
		{#1}}}

\providecommand{\siff}{\Leftrightarrow}
\newcommand{\weakly}{\ensuremath{\:{\rightharpoonup}\:}}

\newcommand{\nnn}{\ensuremath{{n\in{\mathbb N}}}}

\newcommand{\menge}[2]{\big\{{#1}~\big |~{#2}\big\}}
\newcommand{\tmenge}[2]{\{{#1}~|~{#2}\}}

\newcommand{\fenv}[1]%
{\ensuremath{\,\overrightarrow{\operatorname{env}}_{#1}}}
\newcommand{\benv}[1]%
{\ensuremath{\,\overleftarrow{\operatorname{env}}_{#1}}}

\newcommand{\sign}{\ensuremath{\operatorname{sign}}}

\newcommand{\scal}[2]{\left\langle{#1},{#2}  \right\rangle}
\newcommand{\sscal}[2]{\langle{#1},{#2}  \rangle}

\newcommand{\RR}{\ensuremath{\mathbb R}}

\newcommand{\dom}{\ensuremath{\operatorname{dom}}}

\DeclareMathOperator*{\argmin}{argmin}
\newcommand{\prox}{\ensuremath{\operatorname{P}}}
\newcommand{\refl}{\ensuremath{\operatorname{R}}}
\newcommand{\ran}{\ensuremath{{\operatorname{ran}}\,}}
\newcommand{\zer}{\ensuremath{\operatorname{zer}}}

\newcommand{\cdom}{\ensuremath{\overline{\operatorname{dom}}\,}}

\newcommand{\Id}{\ensuremath{\operatorname{Id}}}

\newcommand{\TAB}{\operatorname{T}_{A,B}}

\newcommand{\veet}{\ensuremath{{\scriptscriptstyle\vee}}} 
\newcommand{\oveet}{\ensuremath{{\scriptscriptstyle\ovee}}} 

{\begin{list}{}{%
			\settowidth{\labelwidth}{\textrm{#1~}}%
			\setlength{\leftmargin}{\labelwidth+\labelsep}}}
	{\end{list}}
\usepackage{amsthm}
\usepackage[capitalize,nameinlink]{cleveref}
\crefname{equation}{}{equations}
\crefname{chapter}{Appendix}{chapters}
\crefname{item}{}{items}
\crefname{enumi}{}{}
\newtheorem{theorem}{Theorem}[section]
\newtheorem{lemma}[theorem]{Lemma}

\newtheorem{corollary}[theorem]{Corollary}

\newtheorem{proposition}[theorem]{Proposition}

\newtheorem{example}[theorem]{Example}
\newtheorem{ex}[theorem]{Example}
\newtheorem{fact}[theorem]{Fact}
\newtheorem{remark}[theorem]{Remark}

\providecommand{\abs}[1]{\lvert#1\rvert}
\providecommand{\norm}[1]{\lVert#1\rVert}

\providecommand{\stb}[1]{\left\{#1\right\}}
\providecommand{\innp}[1]{\langle#1\rangle}

\providecommand{\LA}{\Leftarrow}
\providecommand{\RA}{\Rightarrow}

\providecommand{\RR}{\mathbb{R}}

\providecommand{\ran}{\operatorname{ran}}

\providecommand{\dom}{\operatorname{dom}}

\newcommand{\fix}{\ensuremath{\operatorname{Fix}}}

\providecommand{\gr}{\operatorname{gra}}
\providecommand{\gra}{\operatorname{gra}}
\providecommand{\Id}{\operatorname{{ Id}}}

\providecommand{\fady}{\varnothing}

\providecommand{\rras}{\rightrightarrows}

\providecommand{\gr}{\operatorname{gra}}
\providecommand{\fix}{\operatorname{Fix}}
\providecommand{\ran}{\operatorname{ran}}
\providecommand{\rec}{\operatorname{rec}}
\providecommand{\Id}{\operatorname{Id}}

\providecommand{\zer}{\operatorname{zer}}
\providecommand{\R}{{ R}}

\providecommand{\fady}{\varnothing}

\newcommand{\cran}{\ensuremath{\overline{\operatorname{ran}}\,}}

\providecommand{\ri}{\operatorname{ri}}

\providecommand{\RR}{\mathbb{R}}

\definecolor{myblue}{rgb}{.8, .8, 1}

\newcommand*\mybluebox[1]{%
	\colorbox{myblue}{\hspace{1em}#1\hspace{1em}}}

\allowdisplaybreaks 

\begin{document}
	
	%

	\author{
		Heinz H.\ Bauschke\thanks{
			Mathematics, University
			of British Columbia,
			Kelowna, B.C.\ V1V~1V7, Canada. E-mail:
			\texttt{heinz.bauschke@ubc.ca}.}~~~and~
		Walaa M.\ Moursi\thanks{
			Department of Combinatorics and Optimization, 
      University of Waterloo,
      Waterloo, Ontario N2L~3G1, Canada.
			E-mail: \texttt{walaa.moursi@uwaterloo.ca}.}
	}

	\title{\textsf{
			On the Douglas--Rachford algorithm for solving\\
			possibly inconsistent optimization problems
		}
	}

	\date{June 21, 2021}
	
	\maketitle
	
	\begin{abstract}
  More than 40 years ago, Lions and Mercier introduced in a seminal paper the Douglas--Rachford algorithm. Today, this method is well recognized as a classical and 
  highly successful splitting method to find minimizers of the sum of 
  two (not necessarily smooth) convex functions. 
  While the underlying theory has matured, one case 
  remains a mystery: the behaviour of the shadow sequence when the given 
  functions have disjoint domains. 
  
  Building on previous work, we establish for the first time 
  weak and value convergence of the shadow sequence generated by
  the Douglas--Rachford algorithm in a setting of unprecedented generality. 
  The weak limit point is shown to solve
  the associated normal problem which is a minimal perturbation of the original
  optimization problem. We also present new results on the geometry of the minimal displacement vector. 
	\end{abstract}


	\section{Introduction}
	
	\subsection{Problem statement and contribution}
	\label{sec:intromain}
	Throughout, we assume that
	\begin{empheq}[box=\mybluebox]{equation}
		\text{$X$ is
			a real Hilbert space space with inner product
			$\scal{\cdot}{\cdot}\colon X\times X\to\RR$, }
	\end{empheq}
	and induced norm $\|\cdot\|$,
	and that
	\begin{empheq}[box=\mybluebox]{equation}
\label{e:AandB}
		\text{$A$ and $B$ are maximally monotone operators on $X$.}
	\end{empheq}
	We set $\dom A \coloneqq \menge{x\in X}{Ax\neq\fady}$ and
	$\ran A\coloneqq A(X)\coloneqq \bigcup_{x\in X}Ax$.
	Recall that the associated \emph{Douglas--Rachford operator} is
	\begin{empheq}[box=\mybluebox]{equation}
		\label{e:def:T}
		T \coloneqq \TAB \coloneqq \Id -\J{A} + \J{B}\R{A},
	\end{empheq}
	where $\J{A}\coloneqq (\Id+A)^{-1}$ and $\R{A} \coloneqq 2\J{A}-\Id$ are
	the \emph{resolvent} and \emph{reflected resolvent} of $A$, respectively.
	(See, e.g., \cite{BC} for background material and further references.)
	It is well known (see \cite{Svaiter}) that if $A+B$ admits a zero,
	i.e., $Z_0 \coloneqq \zer(A+B)\coloneqq (A+B)^{-1}(0)\neq\fady$,
	then the sequence $(\J{A}T^nx)_\nnn$ converges weakly to a point in $Z_0$.
	This is the celebrated \emph{Douglas--Rachford algorithm} which dates
	back to Douglas and Rachford \cite{DougRach} but whose importance to optimization was
	revealed in the seminal paper by Lions and Mercier \cite{LionsMercier}.
	In particular, if $f$ and $g$ are proper lower semicontinuous convex functions
	on $X$ and $A=\partial f$ and $B=\partial g$, then
	$(\J{A}T^nx)_\nnn$ converges weakly to a minimizer of $f+g$ 
  when $\zer(\partial f+\partial g)\neq\fady$. This  
	explains the importance of the Douglas--Rachford algorithm in optimization.
	
	However, it is natural to inquire what the behaviour of the Douglas--Rachford algorithm is when $Z_0 = \fady$. In fact, this has been the focus of
	recent research on the Douglas--Rachford algorithm and
	the closely related alternating direction method of multipliers;
	see, e.g., \cite{BGSB}, \cite{130}, \cite{LRY}, and \cite{RLY}.
	A key quantity to study this general case is the
	\emph{minimal displacement vector}
	\begin{subequations}
 \label{e:devall}
		\begin{empheq}[box=\mybluebox]{equation}
			\label{e:defv}
			v\coloneqq \Pj{\cran  (\Id-T)}(0),
		\end{empheq}
		i.e., $v$ is the projection of $0$ onto the nonempty closed convex
		subset $\cran(\Id-T)$ of $X$. This vector encodes the minimal perturbation
		of the original problem that makes the sum problem possibly feasible.
		Some of the results we shall prove involve the vectors
		\begin{empheq}[box=\mybluebox]{equation}
			\label{e:def:vD:vR}
			v_D\coloneqq \Pj{\overline{\dom A-\dom B}}(0)
			\;\;\text{and}\;\;
			v_R\coloneqq \Pj{\overline{\ran A+\ran B}}
			(0)
		\end{empheq}
	\end{subequations}
	which are also well defined: indeed, the two sets
	\begin{empheq}[box=\mybluebox]{equation}
		\label{e:DandR}
		D\coloneqq \dom A-\dom B
		\;\;\text{and}\;\;
		R\coloneqq \ran A+\ran B
	\end{empheq}
	have \emph{closures that are convex} because the closures of
	the four sets
	$\dom A$, $\dom B$, $\ran A$, $\ran B$ are already \emph{convex}
	due to the maximal monotonicity of  $A$ and $B$
	(see \cite[Corollary~21.14]{BC}).

	The \emph{goal of this paper} is to substantially advance the
	understanding of the Douglas--Rachford algorithm applied to
	convex optimization problems. Fortunately, our main result can
	be stated elegantly after we introduce some necessary notation:
	Suppose that $f$ and $g$ belong to $\Gamma_0(X)$, i.e., that 
	\begin{equation}
		\text{$f$ and $g$ are convex, lower semicontinuous, and proper on $X$,}
	\end{equation}
	that $(A,B)=(\partial f,\partial g)$ and
	thus $\J{A}$ and $\J{B}$ turn into the \emph{proximal mappings}
	\begin{empheq}[box=\mybluebox]{equation}
		\prox_f \coloneqq (\Id+\partial f)^{-1}
		\;\;\text{and}\;\;
		\prox_g \coloneqq (\Id+\partial g)^{-1},
	\end{empheq}
	with corresponding \emph{reflected proximal mappings}
	$\R{f}\coloneqq 2\prox_f-\Id$ and $\R{g}\coloneqq 2\prox_g-\Id$, respectively.
	Under appropriate assumptions,
	our main result (see \cref{thm:conv} below) states that
	for every $x\in X$, there exists a vector $\overline{z}\in X$ such that
\begin{itemize}
\item \textbf{(minimizer of normal problem)} \;
$f(\overline{z})+g(\overline{z}-v)
=  \displaystyle\min_{y\in X} \big(f(y)+g(y-v)\big)$;
\item \textbf{(shadow convergence)}
\; $\prox_fT^n x\weakly  \overline{z}$ and 
$\prox_g\refl_fT^n x\weakly  \overline{z}-v$;\\[-3mm]
\item \textbf{(value convergence)}\;
$f(\prox_fT^n x)\to f(\overline{z})$ and 
$g(\prox_g\refl_fT^n x)\to g(\overline{z}-v)$. 
\end{itemize}
This beautifully captures the case of finding a minimizer of
	$f+g$ when $v=0$! \emph{Moreover, to the best of our knowledge,
		this is the first time where 
		weak convergence of the shadow sequence $(\prox_f T^nx)_\nnn$ is obtained
		along with value convergence in this generality.}  
	Along our journey to the proof of this result, we discover and present
	various substantial improvements of earlier results in this quite general case.
	
	In the remainder of this section, we provide a brief history
	of previous results and also an outline of the rest of the paper.
	
	\subsection{Brief history for the inconsistent case}
	
	The story begins with the 2003 paper \cite{Lukepaper} concerning 
	two nonempty closed convex subsets $U,V$ of $X$.
The authors of \cite{Lukepaper} proved that  when 
	$(f,g)=(\iota_U,\iota_V)$, then 
	the shadow sequence $(\Pj{U}T^n x)_\nnn$
	is bounded and its weak cluster points
	are minimizers of the function $\iota_U+\iota_V(\cdot-v)$.
	In 2013, and motivated by the results in \cite{Lukepaper},
	the authors of \cite{Sicon} established a powerful static framework to 
	cope with the inconsistent problem
	in the general setting of maximally monotone 
	operators. 
	Another major milestone was the 2014 paper \cite{MOR}
	where the connection between the range of the displacement 
	mapping associated with the Douglas--Rachford operator,
	namely $\Id-T$, and the domains and 
	ranges of the individual operators
	was established.

	Building on \cite{Lukepaper},
	\cite{Sicon}, and \cite{MOR},
  the present authors 
	provided useful convergence results 
	for the shadow sequence 
	of the Douglas--Rachford algorithm in various instances.
	Indeed, 
	the first proof of \emph{strong} convergence
	of the shadow sequence was given when 
	$(f,g)=(\iota_U,\iota_V)$ for two 
	closed affine subspaces $U,V$ of $X$
	in 2015 in \cite{101}.
	Even more strikingly, they obtained linear rates of convergence
	with the rate being quantified in terms of the 
	cosine of the Friedrichs angle between $U$ and $V$.
	In another 2015 paper \cite{BDM:16}, together with M.N.\ Dao,
	they extended
	the result to the setting when one set is a nonempty closed 
	convex (but not necessarily affine) subset of $X$.
	Another milestone is the work in the 
	2016 paper \cite{BM:MPA17} where
	the authors presented a new Fej\'er monotoncity principle
	to prove the full weak convergence of the shadow sequence
	in the case of two nonempty closed convex 
	(not necessarily intersecting) subsets of $X$. 
  This completed the analysis for two indicator functions 
  that began in \cite{Lukepaper} 13 years earlier.
	We refer the interested reader to \cite{Moursithesis} for a 
  detailed collection of the previously mentioned 
	results. 
	The latest breakthrough was the 2019 paper \cite{130} 
  which dealt with the case when 
	$f=\iota_U$, where $U$ 
	is a closed linear subspace of $X$,
	and $g\in\Gamma_0(X)$. 
(In passing, we point out that
		at the time of writing \cite{130} an assumption was made there 
    --- namely $v_R=0$ --- that was
    \emph{sufficient} for convergence.
		In the present paper, we clarify this further 
		by proving that $v_R=0$ 
		is also a \emph{necessary} condition for convergence --- see 
    \cref{prop:necessary:cond} below.)
	
	We now turn to related works that built on the previous results.
	In the 2017  paper \cite{LRY}
	(and the 2018 paper \cite{RLY}) Ryu, Lin, and Yin 
	proposed a method based on the Douglas--Rachford algorithm
	that identifies, in certain situations,  infeasible, unbounded, 
	and pathological conic (and feasible and infeasible convex, respectively) optimization problems.
	In the 2018 paper, they translated the analysis to
	ADMM, which is an incarnation of the Douglas--Rachford algorithm 
	(see, e.g., \cite{Gabay} or \cite{MZinch19}).
	The analysis hinges upon identifying 
	the range of $\Id-T $ and the notion and location 
	of the minimal displacement vector defined in \cite{MOR}.
	Closely related in spirit  to the results in \cite{LRY}
	and  \cite{RLY}
	is the 2018 paper by Banjac, Goulart, Stellato, and Boyd 
	\cite{BGSB}. Indeed, these authors showed that for certain
	classes of convex optimization problems, ADMM can detect
	primal and dual infeasibility of the problem and they 
	propose a termination criterion.
	In the recent 2020 papers \cite{BL21} and \cite{B21} 
	the authors
	extended some of the geometric properties of the 
	minimal displacement vector established in 
	\cite{130}. In particular, the decomposition 
	of $v$ into the sum of orthogonal vectors $v=v_D+v_R$
	(see \cref{e:def:vD:vR}) was seen to be useful in studying certain structured  optimization problems.
	Finally, our work has proven to be useful even in nonconvex settings; 
  indeed, Borwein, Lindstrom, Sims, Schneider, and Skerritt 
  have used the main result of 
	\cite{BM:MPA17} to extend some of the convex theory to the 
	non-convex case in their 2018 paper \cite{BLSSS18}.

	\subsection{Organization of the paper}
	
	The remainder of this paper is organized as follows.
	In \cref{sec:aux}, we collect various results on 
	the normal problem, the solution set, and the set of
	minimizers of the sum of two functions to make subsequent
	proofs easier to follow. 
  Useful properties of $v_D$ and $v_R$ 
  (which are defined in \cref{e:def:vD:vR}) 
  are revealed in \cref{sec:vDandvR}.
In \cref{sec:static}, 
we present results on the interplay 
between the vectors $v,v_D,v_R$, 
the Douglas--Rachford operator $T$, 
and the generalized solution set $Z$. 
We analyze the shadow sequence $(\J{A}T^nx)_\nnn$ with regards to 
Fej\'er monotonicity and conditions necessary for convergence in
\cref{sec:dynamic}. 
Finally, in \cref{sec:yay}, we prove the main result. 
	
	Any terminology and notation not explicitly defined here can be found in 
	\cite{BC}. 
	
	\section{Background material and auxiliary results} 
	
	\label{sec:aux}
	
	In this section, we recall and record various results
	concerning the normal problem, the extended solution set, and
	the set of minimizers of two functions.
	
	Recall first the well-known
	inverse resolvent identity (see, e.g., \cite[Proposition~23.20]{BC})
	\begin{equation}
		\label{e:iri}
		\J{A}+\J{A^{-1}}=\Id
	\end{equation}
	as well as the Minty parametrization (see  \cite{Minty}) 
	\begin{equation}
		\label{Min:par}
		\gra A = \menge{(\J{A} x,x-\J{A} x)}{x\in X}.
	\end{equation}
	The following resolvent identities will be useful later: 
	\begin{lemma}
		\label{lem:bad:walaa:ii}
		Let $(y,w)\in X\times X$. Then
		$\J{A} y=\J{-w+A}(-w+y)$
		and
		$\J{A^{-1}} y=w+\J{(-w+A)^{-1}}(-w+y)$.
	\end{lemma}
	\begin{proof}
		Indeed, recall that $\J{-w+A}=\J{A}(\cdot+w)$
		by, e.g., \cite[Proposition~23.17]{BC}.
		Now, $\J{A} y=\J{A}((-w+y)+w)=\J{-w+A}(-w+y)$
		and
		$\J{A^{-1}} y=y-\J{A}y= y-\J{-w+A}(-w+y)
		=w+(-w+y)-\J{-w+A}(-w+y)
		=w+\J{(-w+A)^{-1}}(-w+y)$.
	\end{proof}
	
	Turning now to 
	the Douglas--Rachford operator $T$ introduced in \cref{e:def:T}, 
	we note that 
	\begin{equation}
		\label{e:Id-T:desc}
		\Id-T = \J{A}-\J{B}\R{A}=\J{A^{-1}}+\J{B^{-1}}\R{A},
	\end{equation}
	and also the following:
	
	\begin{lemma}
		\label{lem:asymp:shadow}
		We have the following:
		\begin{enumerate}
    \setlength\itemsep{0.35em}
			\item
			\label{lem:asymp:shadow:i}
			$\J{A}-\J{A}T=\J{A^{-1}} T+\J{B^{-1}}\R{A}$;
			hence, $\ran(\J{A}-\J{A}T)\subseteq \ran A+\ran B$.
			\item
			\label{lem:asymp:shadow:ii}
			$\J{A^{-1}}-\J{A^{-1}}T=\J{A} T-\J{B}\R{A}$; 
			hence, $\ran(\J{A^{-1}}-\J{A^{-1}}T)\subseteq\dom A-\dom B$.
		\end{enumerate}
	\end{lemma}
	\begin{proof}
		\cref{lem:asymp:shadow:i}:
		Indeed, $\J{A}-\J{B^{-1}}\R{A}
		=\J{A}-\R{A}+\J{B}\R{A}
		=\J{A}-2\J{A}+\Id+\J{B}\R{A}
		=\Id-\J{A}+\J{B}\R{A}
		= T=\J{A}T+\J{A^{-1}} T$.
		Rearranging yields the desired result.
		\cref{lem:asymp:shadow:ii}:
		Indeed,
		$\J{A^{-1}}+\J{B}\R{A}
		=\Id-\J{A}+\J{B}\R{A}
		=T
		=\J{A^{-1}}T+\J{A} T$.
		Rearranging yields the desired result.
	\end{proof}

	Recalling the definitions of the  sets $D$ and $R$ from \cref{e:DandR},
	we have (see \cite[Proposition~4.1]{EckThesis} or \cite[Corollary~2.14]{Sicon})
	\begin{equation}
		\ran(\Id-T) = \menge{a-b}{(a,a^*)\in\gr A,\, (b,b^*)\in \gr B,\, a-b=a^*+b^*}
		\subseteq D \cap R.
	\end{equation}
	It follows that 
	$\cran(\Id-T)\subseteq \overline{D\cap R}\subseteq 
	\overline{D}\cap\overline{R}$.
	From this point onwards, we will assume that 
	\begin{empheq}[box=\mybluebox]{equation}
		\label{e:assump:free:fd}
		\cran(\Id-T)=\overline{D\cap R}
		=\overline{D}\cap \overline{R}.
	\end{empheq}
	For applications, this assumption is rather mild as can been seen
	in the following:

	\begin{remark}
		\label{rem:assump:free:fd}
		It follows from \cite[Corollary~6.5]{MOR}
		that \cref{e:assump:free:fd}
		holds if $X$ is finite-dimensional
		and $A$ and $B$ are subdifferential operators
		of proper  lower semicontinuous  convex functions. 
		See also \cite[Theorem~5.2]{MOR} for  more general settings.
	\end{remark}

	Following \cite{Sicon},
	recall that \emph{the normal problem} associated with the ordered pair
	$(A,B)$ is
	\begin{equation}
		\label{P}
		\text{ find $x\in X$ such that  $0\in -v+Ax+B(x-v)$,}
	\end{equation}
	where $v$ is as in \cref{e:defv}.
	Next, 
	the \emph{Attouch--Th\'{e}ra dual pair} (see \cite{AT}
	and \textcolor{black}{\cite[page~40]{Mercier}}) of
	the
	\emph{primal} pair
	$(-v+A,B(\cdot-v))$
	is $(-v+A,B(\cdot-v))^*:=((-v+A)^{-1},(B(\cdot-v))^{-\ovee})$,
	where
	$(B(\cdot-v))^{\ovee}\coloneqq  (-\Id)\circ B(\cdot-v)\circ(-\Id)$ and 
	$B^{-\ovee}\coloneqq ((B(\cdot-v))^{-1})^\ovee=((B(\cdot-v))^\ovee)^{-1}$.
	We will make use of the notation 
	\begin{empheq}[box=\mybluebox]{equation}
		\label{e:def:Zv}
		Z\coloneqq Z_{(-v+A,B(\cdot-v))}=(-v+A+B(\cdot-v))^{-1}(0)
	\end{empheq}
	and
	\begin{empheq}[box=\mybluebox]{equation}
		\label{e:def:Kv}
		K\coloneqq K_{(-v+A,B(\cdot-v))}=((-v+A)^{-1}+(B(\cdot-v))^{-\ovee})^{-1}(0)
	\end{empheq}
	to denote the \emph{primal} and \emph{dual} solutions
	of the normal problem \cref{P}, respectively
	(see, e.g., \cite{74}).
	It follows from \cite[Proposition~3.2]{Sicon}
	that
	\begin{equation}
		\label{e:T:normal}
		T_{-v+A, B(\cdot-v)}=T(\cdot+v);
	\end{equation}
	moreover, \cite[Proposition~2.24 and Proposition~3.3]{Sicon}
	imply
	\begin{equation}
		\label{e:Z:Fix}
		Z\neq \fady\siff \fix T(\cdot+v)\neq \fady \siff v\in \ran(\Id-T).
	\end{equation}
	We now recall that the 
	\emph{extended solution set} 
	associated with  the normal problem
	\cref{P}
	(see Eckstein and Svaiter's \cite[Section~2.1]{EckSvai08}
	and also \cite[Section~3]{74}) is defined by
	\begin{equation}
		\label{def:ex:sol}
		\SE\coloneqq \SE_{(-v+A,B(\cdot-v))} :=\stb{(z,k)\in
			X\times X~|~-k\in B(z-v), k\in -v+Az}\subseteq Z\times K.
	\end{equation}
	The usefulness of $\SE$ becomes apparent in the next two results:
	
	\begin{fact}
		\label{fact:para:cc}
		Recalling \cref{Min:par} and \cref{e:T:normal},
		we have 
		\begin{equation}
			\label{fact:para:cc:ii}
			\SE = \menge{(\J{-v+A}\times \J{(-v+A)^{-1}}) (y,y)}{y\in \fix T(\cdot+v)}.
		\end{equation}
		If $A$ and $B$ are paramonotone\footnote{Let $C\colon X\rras X$ be monotone.
			Then
			$C$ is \emph{paramonotone}
			if
			[$
			\{(x,u),(y,v)\}\subseteq \gra C$
			and $\innp{x-y,u-v}=0$
			$]
			\RA
			\big\{(x,v),(y,u)\big\}\subseteq \gra C$.
			(For a more detailed discussion and examples of
			paramonotone operators, we refer the reader to \cite{Iusem98}.)},
		then we additionally have:
		\begin{enumerate}
    \setlength\itemsep{0.35em}
			\setcounter{enumi}{0}
			\item
			\label{fact:para:ii:b}
			$\SE=Z\times K$.
			\item
			\label{fact:para:ii:a}
			$\fix T(\cdot+v)=Z+K$.
			
		\end{enumerate}
	\end{fact}
	\begin{proof}
		The identity \cref{fact:para:cc:ii} is \cite[Theorem~4.5]{74}.
		\ref{fact:para:ii:b}\&\ref{fact:para:ii:a}:
		See \cite[Corollary~5.5(ii)\&(iii)]{74}.
	\end{proof}

	\begin{lemma}
		\label{lem:bad:walaa}
		The following hold:
		\begin{enumerate}
    \setlength\itemsep{0.35em}
			\item
			\label{lem:bad:walaa:i}
			$\fix T(\cdot+v)=-v+\fix( v+T)$.
			\item
			\label{lem:bad:walaa:iii}
			$\SE=(0,-v)+\menge{(\J{A}\times \J{A^{-1}})(f,f)}{f\in \fix (v+T)}.
			$
		\end{enumerate}
	\end{lemma}
	\begin{proof}
		\cref{lem:bad:walaa:i}:
		Let $f\in X$.
		Then $f\in \fix (v+T)
		\siff
		f=v+Tf
		\siff
		f-v=T(f-v+v)
		\siff
		f-v\in \fix T(\cdot+v)
		\siff
		f\in v+\fix T(\cdot+v)
		$.
		
		\cref{lem:bad:walaa:iii}:
		Combine \cref{lem:bad:walaa:i}
		and \cref{lem:bad:walaa:ii}
		to learn that
		$\tmenge{(\J{A}\times \J{A^{-1}})(f,f)}{f\in \fix (v+T)}
		=(0,v)+\tmenge{(\J{-v+A}\times \J{(-v+A)^{-1}})(f,f)}{f\in \fix T(\cdot+v)}
		$.
		Now invoke \cref{fact:para:cc:ii}.
	\end{proof}
	
	We conclude this section with the following useful results
	concerning the minimizers of
	the sum of two functions.
	
	\begin{lemma}
		\label{lem:subgd:fg}
		Let $f\in \Gamma_0(X)$
		and let $g\in \Gamma_0(X)$ be such that
		$\argmin (f+g)\neq \fady$. Let $x\in X$
		and let $y\in \argmin (f+g)$.
		Suppose that  $x^*\in (-\partial f(x))\cap \partial g(x)$.
		Then
		\begin{subequations}
			\label{e:subgrad:scal}
			\begin{align}
				\label{e:subgrad:scal:f}
				f(y) & =f(x)+\scal{-x^*}{y-x},
				\\
				g(y) & =g(x)+\scal{x^*}{y-x},
				\label{e:subgrad:scal:g}
			\end{align}
		\end{subequations}
		and
		\begin{equation}
			\label{lem:subgd:fg:i}
			x^*\in(-\partial f(y))\cap \partial g(y).
		\end{equation}
		If $\iota_C\in \{f,g\}$,
		where $C$ is nonempty closed  convex subset of $X$, 
		then we also have the following:
		\begin{enumerate}
    \setlength\itemsep{0.35em}
			\setcounter{enumi}{0}
			\item
			\label{lem:subgd:fg:ii}
			$\scal{x^*}{y-x}=0$.
			\item
			\label{lem:subgd:fg:iii}
			$K\perp (Z-Z)$.
			\item
			\label{lem:subgd:fg:iv}
			$\J{A}\Pj{\fix T}=\Pj{Z}$.
		\end{enumerate}
	\end{lemma}
	\begin{proof}
		Observe that
		$0=-x^*+x^*\in \partial f(x) +\partial g(x)\subseteq
		\partial(f+g)(x)$.
		Hence $x$ is a minimizer of $f+g$.
		Consequently $f(x)+g(x)=f(y)+g(y)$; equivalently,
		\begin{equation}
			\label{e:squeeze}
			f(x)-f(y)=g(y)-g(x).
		\end{equation}
		The subgradient inequalities for
		$f$ and $g$ yield   $(\forall z\in X)$
		\begin{subequations}
			\label{e:subgrad}
			\begin{align}
				f(z) & \ge f(x)+\scal{-x^*}{z-x}, 
				\\
				g(z) & \ge g(x)+\scal{x^*}{z-x}.
			\end{align}
		\end{subequations}
		In particular, we learn that
		\begin{subequations}
			\label{e:subgd:fg}
			\begin{align}
				f(y) & \ge f(x)-\scal{x^*}{y-x}, 
				\\
				g(y) & \ge g(x)+\scal{x^*}{y-x}.
			\end{align}
		\end{subequations}
		Hence 
		\begin{equation}
			f(x)-f(y)\le\scal{x^*}{y-x}\le g(y)-g(x).
		\end{equation}
		Combining the above inequality with
		\cref{e:squeeze} yields
		\begin{equation}
			\label{e:subgrad:ortho}
			f(x)-f(y)=\scal{x^*}{y-x}= g(y)-g(x).
		\end{equation}
		\cref{lem:subgd:fg:i}:
		Let $z\in X$.
		Then \cref{e:subgrad}, \cref{e:subgrad:ortho}, 
		and \cref{e:subgrad:scal}
		yield   $(\forall z\in X)$
		\begin{subequations}
			\begin{align}
				f(z) & \ge f(x)+\scal{-x^*}{z-x}
				=\underbrace{f(x)+\scal{-x^*}{y-x}}_{=f(y)}
				+\scal{-x^*}{z-y}=f(y)+\scal{-x^*}{z-y},
				\\
				g(z) & \ge g(x)+\scal{x^*}{z-x}
				=\underbrace{g(x)+\scal{x^*}{y-x}}_{=g(y)}
				+\scal{x^*}{z-y}=g(y)+\scal{x^*}{z-y}.
			\end{align}
		\end{subequations}
		Consequently we learn that
		$x^{*}\in -\partial f(y)$
		and
		$x^{*}\in \partial g(y)$.
		\cref{lem:subgd:fg:ii}:
		Suppose first that $g=\iota_C$.
		Because $x$ and $y$ are minimizers of $f+\iota_C$
		we must have $\{x,y\}\subseteq C$, hence
		$\iota_C(x)=\iota_C(y)=g(x)=g(y)=0$.
		Now combine with \cref{e:subgrad:ortho}.
		\cref{lem:subgd:fg:iii}:
		It follows from \cite[Remark~5.4]{74}
		that
		$K=(-\partial f(x))\cap \partial g(x)$.
		Now combine with \cref{lem:subgd:fg:ii}.
		The proof when $f=\iota_C$ is similar.
		\cref{lem:subgd:fg:iv}:
		Combine \cref{lem:subgd:fg:iii}
		and
		\cite[Theorem~6.7(ii)~and~Corollary~5.5(iii)]{74}.
	\end{proof}

	\begin{proposition}
		\label{prop:criticals:mini}
		Let $f\in \Gamma_0(X)$
		and let $g\in \Gamma_0(X)$ be such that
		$\zer(\partial f+\partial g)\neq \fady$.
		Then
		\begin{equation}
			\argmin (f+g)=  \zer \partial(f+g)=\zer(\partial f+\partial g).
		\end{equation}
	\end{proposition}
	\begin{proof}
		Observe that
		$\zer(\partial f+\partial g)\subseteq \zer \partial(f+g)=\argmin (f+g)$
		by, e.g., \cite[Theorem~16.3~\&~Proposition~16.6(ii)]{BC}.
		It remains to establish the inclusion
		$\zer \partial(f+g)\subseteq\zer(\partial f+\partial g) $.
		To this end, let $y\in \zer \partial(f+g)=\argmin (f+g)$
		and let $x\in \zer(\partial f+\partial g)$.
		Then $(\exists x^*\in (-\partial f(x))\cap \partial g(x))$.
		Using \cref{lem:subgd:fg:i}, we learn that
		$x^*\in (-\partial f(y))\cap \partial g(y)$, hence
		$0\in \partial f(y)+\partial g(y)$.
		Consequently, $y\in \zer(\partial f+\partial g)$.
	\end{proof}
	
	\section{$v_D$ and $v_R$}
	
	\label{sec:vDandvR}
	
In this section, we shall derive various results on 
the vectors $v_D$ and $v_R$ (see \cref{e:def:vD:vR}).
Our analysis depends on the following two results.

	\begin{fact}
		\label{f:geo2sets}
		Let $U$ and $V$ be nonempty closed convex subsets of $X$.
		Then
		\begin{equation}
			\Pj{\overline{U-V}}(0)\in\overline{(\Pj{U}-\Id)(V)}\cap\overline{(\Id-\Pj{V})(U)}
			\subseteq
			(-\rec U)^\ominus \cap(\rec V)^\ominus.
		\end{equation}
	\end{fact}
	\begin{proof}
		This follows from
		\cite[Corollary~2.7]{Lukepaper} and \cite[Theorem~3.1]{Zara}.
	\end{proof}
	
	\begin{lemma}
		\label{lem:rec:dom:ran}
		The following hold for $A$ and $B$ (see \cref{e:AandB}):
		\begin{enumerate}
			\item
			\label{lem:rec:dom:ran:i}
			$(\rec \cdom  A)^\ominus\subseteq \rec (\cran   A)$ 
			and $(\rec \cdom  B)^\ominus\subseteq \rec (\cran   B)$.\\[-3mm]
			\item
			\label{lem:rec:dom:ran:ii}
			$(\rec {\cran} A)^\ominus\subseteq \rec (\cdom  A)$
			and $(\rec {\cran} B)^\ominus\subseteq \rec (\cdom  B)$.
		\end{enumerate}
	\end{lemma}
	
	\begin{proof} It suffices to prove the statements for $A$. 
		\cref{lem:rec:dom:ran:i}:
		Observe that using, e.g., \cite[Corollary~21.14~and~Example~25.14]{BC}
		$\cdom  A$ and $\cran A$ are nonempty closed and
		convex subsets of $X$, that $\N{\cdom  A}$
		is $3^*$ monotone and maximally monotone and that    
		$A=A+\N{\cdom  A}$.
		On the one hand,
		it follows from Brezis--Haraux theorem (see, e.g., \cite[Theorem~25.24(ii)]{BC})
		applied to $A$ and $\N{\cdom  A}$ that
		\begin{equation}
			{\cran}A+{\cran}\N{\cdom  A}
			\subseteq \overline{\ran A+\ran \N{\cdom  A}}
			={\cran}(A+\N{\cdom  A})
			={\cran}A.
		\end{equation}
		Hence
		\begin{equation}
			\label{e:210522:a}
			\cran  \N{\cdom  A}
			\subseteq
			\rec \cran  A.
		\end{equation}
		On the other hand, it follows from \cite[Theorem~3.1]{Zara}
		that
		\begin{equation}
			\label{e:210522:b}
			\cran  \N{\cdom  A}
			=\cran  (\Id-\Pj{\cdom  A})
			=(\rec \cdom  A)^\ominus.
		\end{equation}
		Now combine \cref{e:210522:a}
		and \cref{e:210522:b}.
		\cref{lem:rec:dom:ran:ii}:
		Apply   \cref{lem:rec:dom:ran:i}
		to $A^{-1}$.
	\end{proof}

We are now able to derive new information about the location
of $v_D$ and $v_R$. When we specialize to subdifferential operators, 
then one obtains a recent result 
(see 
\cite[Proposition~3.2]{B21})
that was proved differently 
by using recession functions 
(which are unavailable in general). 
These results in turn generalize
\cite[Proposition~2.3]{130}. 
	
	\begin{proposition}
		\label{prop:v:decom}
		The following hold:
		\begin{enumerate}
    \setlength\itemsep{0.35em}
			\item
			\label{prop:v:decom:i}
			$v_D\in (-\rec \cdom  A)^\ominus\cap (\rec \cdom  B)^\ominus
			=(-(\rec \cdom  A)^\ominus)\cap (\rec \cdom  B)^\ominus$.
			\item
			\label{prop:v:decom:ii}
			$v_D\in (-\rec \cran   A)\cap (\rec \cran   B)$.
			\item
			\label{prop:v:decom:iii}
			$v_R\in  (-\rec \cran   A)^\ominus\cap (-\rec \cran   B)^\ominus
			=-((\rec \cran   A)^\ominus\cap (\rec \cran   B)^\ominus)$.
			\item
			\label{prop:v:decom:iv}
			$v_R\in(-\rec \cdom  A)\cap (-\rec \cdom  B)
			=-(\rec \cdom  A\cap\rec \cdom  B)$.
			\item
			\label{prop:v:decom:v}
			$\scal{v_D}{v_R}=0$.
			\item
			\label{prop:v:decom:vi}
			${v_D}+{v_R}\in \overline{\dom A-\dom B}\cap\overline{\ran A+\ran B}$.
			\item
			\label{prop:v:decom:viii}
			$v=v_D+v_R $.
			\item
			\label{prop:v:decom:ix}
			$\norm{v}^2=\norm{v_D}^2+\norm{v_R}^2   =\norm{(v_R,v_D)}^2$.
			
		\end{enumerate}
	\end{proposition}
	
	\begin{proof}
		\cref{prop:v:decom:i}:
		Apply \cref{f:geo2sets}
		with $(U,V)$ replaced by
		$(\cdom  A,\cdom  B)$.
		
		\cref{prop:v:decom:ii}:
		Combine \cref{prop:v:decom:i}
		and \cref{lem:rec:dom:ran}\cref{lem:rec:dom:ran:i}.
		
		\cref{prop:v:decom:iii}:
		Apply \cref{f:geo2sets}
		with $(U,V)$ replaced by
		$(\cran   A,-\cran   B)$.
		
		\cref{prop:v:decom:iv}:
		Combine \cref{prop:v:decom:iii}
		and \cref{lem:rec:dom:ran}\cref{lem:rec:dom:ran:ii}.
		
		\cref{prop:v:decom:v}:
		It follows from \cref{prop:v:decom:i} and \cref{prop:v:decom:iv}
		that $(-v_D,-v_R)\in (\rec\cdom  A)^\ominus\times \rec\cdom  A$.
		Hence $\scal{v_D}{v_R}=\scal{-v_D}{-v_R}\le 0$.
		Similary,  \cref{prop:v:decom:ii} and \cref{prop:v:decom:iii}
		imply that
		$(v_D,-v_R)\in \rec\cran   B\times(\rec\cran   B)^\ominus$.
		Hence, $-\scal{v_D}{v_R}=\scal{v_D}{-v_R}\le 0$.
		Altogether, $\scal{v_D}{v_R}=0$.
		
		\cref{prop:v:decom:vi}:
		Indeed, in view of \cref{prop:v:decom:iv}
		we have $-v_R\in \rec \cdom  B$.
		Therefore,
		${v_D}+{v_R}
		\in \overline{\dom A-\dom B}+{v_R}
		=\overline{\dom A-(-v_R+\cdom B)}
		\subseteq  \overline{\dom A-\cdom  B}
		=\overline{\dom A-\dom B}$.
		Similarly, in view of \cref{prop:v:decom:ii}
		we have $v_D\in \rec \cran   B$.
		Therefore
		${v_D}+{v_R}
		\in\overline{\ran A+v_D+\cran B}
		\subseteq \overline{\ran A+\cran   B}
		=
		\overline{\ran A+{\ran} B}
		$.
		
		\cref{prop:v:decom:viii}:
		Observe that
		\cref{prop:v:decom:vi}
		and \cref{e:assump:free:fd}
		imply that
		$\norm{v}\le\norm{v_D+v_R}$.
		It follows from \cref{prop:v:decom:v},
		the definition of $v$ and $v_D$
		that
		$\norm{v_D}^2\le \scal{v_D}{\overline{D}}$,
		hence
		$\norm{v_D}^2\le \scal{v_D}{\overline{D}\cap \overline{R}}$.
		Similarly,
		\cref{prop:v:decom:v},
		the definition of $v$ and
		and $v_R$
		implies
		$\norm{v_R}^2\le \scal{v_R}{\overline{R}}$,
		hence
		$\norm{v_R}^2\le \scal{v_R}{\overline{D}\cap \overline{R}}$.
		Therefore using
		Cauchy-Schwarz
		we learn that
		$\norm{v_D+v_R}^2
		=\norm{v_D}^2+\norm{v_R}^2
		\le \scal{v}{v_D}+\scal{v}{v_R}
		=\scal{v}{v_D+v_R}\le \norm{v}\norm{v_D+v_R}$.
		Hence, $\norm{v_D+v_R}\le \norm{v}$.
		Altogether, $\norm{v}=\norm{v_D+v_R}$.
		In view of \cref{e:assump:free:fd},
		\cref{prop:v:decom:vi}
		and \cite[Lemma~2]{Pazy1970},
		we learn that
		$v=v_D+v_R$.
		
		\cref{prop:v:decom:ix}:
		Combine
		\cref{prop:v:decom:v}
		and \cref{prop:v:decom:viii}.
	\end{proof}
	
	\begin{remark}[{\bf The real line case}]
		Suppose that $X=\RR$.
		It follows from \cref{prop:v:decom}\cref{prop:v:decom:v}
		that $v_Dv_R=0$ which implies that
		\begin{equation}
			0\in \{v_D,v_R\}.
		\end{equation}
		This conclusion
		is no longer true when $X\neq \RR$ as we illustrate in
		\cref{ex:not:eq:Z} and \cref{ex:eq:Z} below.
	\end{remark}
  
The proof of the last result in this section requires the fact that 
if  $C$ is a nonempty closed convex subset of $X$ and $x\in X$, then 
(see, e.g., \cite[Proposition~29.1(iii)]{BC})
	\begin{equation}
		\label{e:cones:-}
		\Pj{-C}(x)=-\Pj{C}(-x).
	\end{equation}
	
	\begin{corollary}
		\label{cor:relate:vs}
		The following hold: 
		\begin{enumerate}
			\item
			\label{cor:relate:vs:i}
			$v_D=\Pj{(\rec \cdom A)^\oplus}v$.
			\item
			\label{cor:relate:vs:ii}
			$v_R=\Pj{-\rec \cdom A}v$.
			\item
			\label{cor:relate:vs:iii}
			$v_D=\Pj{-\rec \cran A}v$.
			\item
			\label{cor:relate:vs:iv}
			$v_R=\Pj{(\rec \cran A)^\oplus}v$.
		\end{enumerate}
	\end{corollary}
	\begin{proof}
		Observe that $\rec \cdom A$
		and $\rec\cran A$ are closed by, e.g., \cite[Proposition~6.49(v)]{BC}.
		\cref{cor:relate:vs:i}\&\cref{cor:relate:vs:ii}:
		It follows from
		\cref{prop:v:decom}\cref{prop:v:decom:i},\cref{prop:v:decom:iv},\cref{prop:v:decom:v}\&\cref{prop:v:decom:viii}
		that $(-v_D,-v_R)\in (\rec \cdom A)^\ominus \times \rec \cdom A$,
		that $v_D\perp v_R$ and that $-v_D=-v-(-v_R)$.
		Now combine with \cite[Proposition~6.28]{BC} to learn that
		$-v_D=\Pj{(\rec \cdom A)^\ominus}(-v)$; equivalently,
		$v_D=-\Pj{(\rec \cdom A)^\ominus}(-v)=\Pj{-(\rec \cdom A)^\ominus}v
		=\Pj{(\rec \cdom A)^\oplus}v$, where the third identity followed from applying
		\cref{e:cones:-} with $(x,C)$ replaced by $(v,(\rec \cdom A)^\ominus)$.
		Similarly,
		$-v_R=\Pj{\rec \cdom A}(-v)$; equivalently,
		$v_R=-\Pj{\rec \cdom A}(-v)=\Pj{-\rec \cdom A}v$,
		where the last identity followed from applying
		\cref{e:cones:-} with $(x,C)$ replaced by $(v,\rec \cdom A)$.
		
		\cref{cor:relate:vs:iii}\&\cref{cor:relate:vs:iv}:
		It follows from
		\cref{prop:v:decom}\cref{prop:v:decom:ii},\cref{prop:v:decom:iii},\cref{prop:v:decom:v}\&\cref{prop:v:decom:viii}
		that $(-v_D,-v_R)\in \rec \cran A \times (\rec \cran A)^\ominus$,
		that $v_D\perp v_R$, and that $-v_D=-v-(-v_R)$.
		Now proceed similar to the proof of \cref{cor:relate:vs:i}\&\cref{cor:relate:vs:ii}.
	\end{proof}
	
\section{Static consequences}
 
\label{sec:static}
 
In this section, we present results on the interplay 
between the vectors $v,v_D,v_R$ (see \cref{e:devall}),
the Douglas--Rachford operator $T$ (see \cref{e:def:T}), 
and  the generalized solution set $Z$ (see \cref{e:def:Zv}).
Working in the  product space $X\times X$, 
  we restate \cref{e:def:vD:vR}
	(see, e.g., \cite[Proposition~29.4]{BC}) as:
	\begin{equation}
		\label{e:proj:ps}
		\Pj{\overline{\dom A-\dom B}\times \overline{\ran A+\ran B}}(0)
		=(v_D,v_R).
	\end{equation}
The next result relates $(v_D,v_R)$ to the Douglas--Rachford operator $T$ 
defined in \cref{e:def:T}:
  
	\begin{lemma}
		\label{lem:pd:sh:bdd:iii:c}
		The following hold:
		\begin{enumerate}
			\item
			\label{prop:pd:sh:bdd:iii:c}
			Suppose that $f\in \fix (v+T)$. Then
			\begin{subequations}
				\begin{align}
					Tf           & =f-v,
					\label{e:Tf:f}
					\\
					\J{A}Tf        & =\J{A}f-v_R,
					\label{e:JA:Tf:f}
					\\
					\J{A^{-1}}Tf & =\J{A^{-1}}f-v_D.
					\label{e:JAinv:Tf:f}
				\end{align}
			\end{subequations}
			Moreover,
\begin{subequations}
\label{e:ortho}
\begin{align}
\scal{\J{A}f-\J{A}Tf}{\J{A^{-1}}f-\J{A^{-1}}Tf} \label{e:ortho1}
&=\scal{\J{B}\R{A}f-\J{B}\R{A}Tf}{\J{B^{-1}}\R{A}f-\J{B^{-1}} \R{A}Tf}\\
&=0. \label{e:ortho2}
\end{align} 
\end{subequations}
        \item
			\label{prop:pd:sh:bdd:v}
			$ v\in \ran(\Id-T)\siff \fix (v+T)\neq \fady$ 
      $\RA$ $(v_D,v_R)\in D\times R$ (which were defined in \cref{e:DandR}).
		\end{enumerate}
		
	\end{lemma}

	\begin{proof}
		\cref{prop:pd:sh:bdd:iii:c}:
		\cref{e:Tf:f} is clear.
		Note that \cref{lem:asymp:shadow}\cref{lem:asymp:shadow:i}\&\cref{lem:asymp:shadow:ii}
		applied with $x$ replaced by $f$ yields
		\begin{equation}
			\label{e:210602:a}
			\text
			{$\J{A}f-\J{A}Tf\in \ran A+\ran B~~$
				and
				$~~\J{A^{-1}}f-\J{A^{-1}}Tf\in \dom A-\dom B$.}
		\end{equation}
		In view of \cref{prop:v:decom}\cref{prop:v:decom:ix},
		the Minty parametrization of $\gra A$
		\cref{Min:par},
		and \cref{e:Tf:f}
		imply
		\begin{subequations}
			\begin{align}
				\norm{v_R}^2+\norm{v_D}^2
				& =\norm{\J{A}f-\J{A}Tf}^2
				+\norm{\J{A^{-1}}f-\J{A^{-1}}Tf}^2
				\\
				& \le \norm{\J{A}f-\J{A}Tf}^2
				+\norm{\J{A^{-1}}f-\J{A^{-1}}Tf}^2
				\\
				& \qquad+2\underbrace{\scal{\J{A}f-\J{A}Tf}{\J{A^{-1}}f-\J{A^{-1}}Tf}}_{\ge 0}
				\nonumber\\
				& =\norm{f-Tf}^2
				=\norm{f-(f-v)}^2= \norm{v}^2=\norm{v_R}^2+\norm{v_D}^2.
			\end{align}
		\end{subequations}
		Hence all inequalities become equalities and therefore
		by definition of $v_D$ and $v_R$,
		in view of \cref{e:proj:ps}
		and \cref{e:ortho},
		we must have
		\begin{equation}
			\label{e:210602:b}
			(\J{A}f-\J{A}Tf,\J{A^{-1}}f-\J{A^{-1}}Tf)=(v_R,v_D).
		\end{equation}
		This proves \cref{e:JA:Tf:f} and \cref{e:JAinv:Tf:f}.
    
On the one hand, 
$v=f-Tf=\J{A}f-\J{B}\R{A} f=(\Id-T)Tf= \J{A}Tf-\J{B}\R{A}T f$, hence
$\J{A} f-\J{A} Tf=\J{B}\R{A} f-\J{B}\R{A} Tf$. 
On the other hand, we similarly get 
$v=f-Tf=\J{A^{-1}}f+\J{B^{-1}}\R{A} f=(\Id-T)Tf
=\J{A^{-1}}Tf+\J{B^{-1}}\R{A}T f$,
hence
$\J{A^{-1}} f-\J{A^{-1}} Tf=\J{B^{-1}}\R{A} f-\J{B^{-1}}\R{A} Tf$.
Altogether, \cref{e:ortho1} holds.
Moreover, \cref{e:ortho2} follows from \cref{e:ortho1}, 
\cref{e:210602:b}, and 
\cref{prop:v:decom}\cref{prop:v:decom:v}.

		\cref{prop:pd:sh:bdd:v}:
		Indeed, $\fix (v+T)\neq \fady\siff (\exists x\in X)$
		$x-Tx=v$.
		Now combine \cref{e:210602:a} and \cref{e:210602:b}.
	\end{proof}

In the following result, we relate $v$ to $(v_D,v_R)$: 

	\begin{proposition}
		\label{prop:Zv:2shifts}
		Let $\alpha\ge 0$,
		let $\beta\le 0$
		and let $(x,x^*)\in X\times X$.
		Suppose that $x^*\in (v-Ax)\cap B(x-v)$.
		Then the following hold:
		\begin{enumerate}
    \setlength\itemsep{0.35em}
			\item
			\label{prop:Zv:2shifts:ii}
			$x^*+\alpha v_D\in (v-Ax)\cap  B(x-v)$.
			\item
			\label{prop:Zv:2shifts:iii}
			$x+\beta v_R\in A^{-1}(-x^*+v)\cap (v+B^{-1}x^*)$.
			\item
			\label{prop:Zv:2shifts:iv}
			$x^*\in (v-\alpha v_D-Ax)\cap B(x+\beta v_R-v)$.
		\end{enumerate}
	\end{proposition}
	\begin{proof}
		Let $(a,a^*)\in \gra A$ and let $(b,b^*)\in \gra B$.
		\cref{prop:Zv:2shifts:ii}
		On the one hand, observe that
		$v-x^*\in Ax$. Moreover,
		\begin{subequations}
			\begin{align}
				\scal{a-x}{a^*-(v-x^*-\alpha v_D)}
				& =\scal{a-x}{a^*-(v-x^*)}+\scal{a-x}{\alpha v_D}
				\label{e:vd:a:i}
				\\
				& \ge \scal{a-x}{\alpha v_D}=\alpha\scal{v-(a-(x-v))}{0- v_D}
				\label{e:vd:a:ii}
				\\
				& =\alpha\scal{v_D-(a-(x-v))}{0- v_D}\ge 0.
				\label{e:vd:a:iii}
			\end{align}
		\end{subequations}
		Here \cref{e:vd:a:ii} follows from
		the monotonicity of $A$
		and
		\cref{e:vd:a:iii}
		follows from combining
		\cref{prop:v:decom}\cref{prop:v:decom:viii}\&\cref{prop:v:decom:v}
		and \cref{e:def:vD:vR}
		by observing that $(a,x-v)\in \dom A\times \dom B$.
		The maximality of $A$ implies that
		\begin{equation}
			v-x^*-\alpha v_D\in Ax.
		\end{equation}
		On the other hand, observe that
		$x^*\in B(x-v)$.
		Moreover,
		\begin{subequations}
			\begin{align}
				\scal{b-(x-v)}{b^*-(x^*+\alpha v_D)}
				& =\scal{b-(x-v)}{b^*-x^*}+\alpha \scal{b-(x-v)}{0 -v_D}
				\label{e:vd:b:i}
				\\
				& \ge \alpha \scal{b-(x-v)}{0- v_D}
				\label{e:vd:b:ii}
				\\
				& =\alpha\scal{v_D-(x-b)}{0- v_D}\ge 0.
				\label{e:vd:b:iii}
			\end{align}
		\end{subequations}
		Here \cref{e:vd:b:ii} follows from
		the monotonicity of $B$
		and
		\cref{e:vd:b:iii}
		follows from
		combining
		\cref{prop:v:decom}\cref{prop:v:decom:viii}\&\cref{prop:v:decom:v}
		and \cref{e:def:vD:vR}
		by observing that $(x,b)\in \dom A\times \dom B$.
		The maximality of $B$ implies that
		\begin{equation}
			x^*+\alpha v_D\in B(x-v).
		\end{equation}
		Altogether, we conclude that $x^*+\alpha v_D\in (v-Ax)\cap  B(x-v)$.
		
		\cref{prop:Zv:2shifts:iii}:
		On the one hand, observe that
		$x\in A^{-1}(v-x^*) \cap (v+B^{-1}x^*)$.
		Therefore
		\begin{subequations}
			\begin{align}
				\scal{a^*-(v-x^*)}{a-(x+\beta v_R)}
				& =\scal{a^*-(v-x^*)}{a-x}+\beta \scal{a^*-(v-x^*)}{0-v_R}
				\label{e:vr:a:i}
				\\
				& \ge \beta \scal{a^*-(v-x^*)}{0-v_R}
				\label{e:vr:a:ii}\\
				&=\beta \scal{a^*+x^*-v}{0-v_R}
				\\
				& =\beta \scal{{a^*+x^*}-v_R}{0-v_R}\ge 0.
				\label{e:vr:a:iii}
			\end{align}
			Here \cref{e:vr:a:ii} follows from
			the monotonicity of $A^{-1}$
			and
			\cref{e:vr:a:iii}
			follows from combining
			\cref{prop:v:decom}\cref{prop:v:decom:viii}\&\cref{prop:v:decom:v}
			and \cref{e:def:vD:vR}
			by observing that $(a^*,x^*)\in \ran A\times \ran B$.
		\end{subequations}
		The maximality of  $A^{-1}$ implies that
		\begin{equation}
			x+\beta v_R\in A^{-1}{(v-x^*)}.
		\end{equation}
		On the other hand, because $x\in v+B^{-1}x^*$
		we have
		\begin{subequations}
			\begin{align}
				\scal{b^*-x^*}{b-(x+\beta v_R-v)}
				& =\scal{b^*-x^*}{b-(x-v)}+\beta \scal{b^*-x^*}{0-v_R}
				\label{e:vr:b:i}
				\\
				& \ge \beta \scal{b^*+v-x^*-v}{0-v_R}
				\label{e:vr:b:ii}\\
				&=\beta \scal{b^*+v-x^*-v_D-v_R}{0-v_R}
				\\
				& =\beta \scal{b^*+v-x^*-v_R}{0-v_R}
				\ge 0.
				\label{e:vr:b:iii}
			\end{align}
			Here \cref{e:vr:b:ii} follows from
			the monotonicity of $B$
			and
			\cref{e:vr:b:iii}
			follows from combining
			\cref{prop:v:decom}\cref{prop:v:decom:viii}\&\cref{prop:v:decom:v}
			and \cref{e:def:vD:vR}
			by observing that $(-x^*+v,b^*)\in \ran A\times \ran B$.
		\end{subequations}
		The maximality of $B^{-1}$ implies that
		\begin{equation}
			x+\beta v_R-v\in B^{-1}{x^*}.
		\end{equation}
		\cref{prop:Zv:2shifts:iv}:
		It follows from \cref{prop:Zv:2shifts:ii}
		that $x^*\in v-\alpha v_D-Ax$.
		Moreover, \cref{prop:Zv:2shifts:iii}
		implies that
		$x+\beta v_R-v\in B^{-1}x^*$;
		equivalently,
		$x^*\in B(x+\beta v_R-v)$.
		Altogether, we obtain 
		$x^*\in (v-\alpha v_D-Ax)\cap B(x+\beta v_R-v)$.
	\end{proof}
	
	We are now ready for the following  powerful result about the generalized 
	solution
	set $Z$, which we defined in \cref{e:def:Zv} to be 
	\begin{equation}
		Z=\menge{x\in X}{0\in -v+Ax+B(x-v)}.
	\end{equation}
This result provides a useful alternative description of $Z$ when $v_R=0$.

	\begin{theorem}
		\label{prop:Zvv:2shifts}
		Set 
    \begin{equation} 
    \widetilde{Z}\coloneqq \menge{x\in X}{0\in -v_R+Ax+B(x-v_D)}.
    \end{equation}
		Then the following hold:
		\begin{enumerate}
    \setlength\itemsep{0.35em}
			\item
			\label{prop:Zvv:2shifts:i}
			$  \widetilde{Z}\subseteq Z.$
			\item
			\label{prop:Zvv:2shifts:ii}
			Suppose that $v_R=0$. Then $  \widetilde{Z}= Z.$
		\end{enumerate}
	\end{theorem}
	\begin{proof}
		\cref{prop:Zvv:2shifts:i}:
		Suppose that $x\in \widetilde{Z}$.
		Then $(\exists x^*\in X)$ such that
		$x^*\in (-Ax)\cap(-v_R +B(x-v_D))$.
		Let $(b,b^*)\in \gra B$.  We have
		\begin{subequations}
			\begin{align}
				\scal{b-(x-v)}{b^*-(x^*+v)}
				& =\scal{b-(x-v_D-v_R)}{b^*-(x^*+v_D+v_R)}
				\label{e:opp:inc:-1}\\
				& =\scal{b-(x-v_D)}{b^*-(x^*+v_R)}
				\label{e:opp:inc:0}\\
				&\qquad +\scal{b-(x-v_D)}{0-v_D} 
				+\scal{v_R}{b^*-(x^*+v_R)} \nonumber \\
				& \ge \scal{v_D-(x-b)}{0-v_D}+\scal{v_R}{b^*-x^*-v_R}
				\label{e:opp:inc}\\
        &\geq 0 \label{e:opp:inc+}
			\end{align}
		\end{subequations}
where 
\cref{e:opp:inc:-1} follows from 
\cref{prop:v:decom}\cref{prop:v:decom:viii},
where	\cref{e:opp:inc:0} follows from \cref{prop:v:decom}\cref{prop:v:decom:v},
where \cref{e:opp:inc}
		follows from the monotonicity
		of $B$ and   \cref{prop:v:decom}\cref{prop:v:decom:viii},
    and 
where \cref{e:opp:inc+} follows from
		definitions of $v_D$
		and $v_R$
		by observing that
		$(x,b)\in \dom A\times \dom B$
		and $(-x^*,b^*)\in \ran B\times \ran A$.
		The maximality of $B$ yields
		$x^*+v\in B(x-v)$.
		Recalling that
		$-x^*\in Ax$ we learn that $x\in Z$.
		
		\cref{prop:Zvv:2shifts:ii}:
		In view of \cref{prop:Zvv:2shifts:i} it is sufficient
		to prove the inclusion $Z\subseteq  \widetilde{Z}$.
		To this end, let $x\in Z$. Observe that
		\cref{prop:v:decom}\cref{prop:v:decom:viii}
		implies that $v=v_D$.
		Then $(\exists x^*\in X)$ such that
		$x^*\in (v-Ax)\cap B(x-v)$.
		It follows from \cref{prop:Zv:2shifts}\cref{prop:Zv:2shifts:iv}
		applied with
		$\alpha =1$ and \cref{prop:v:decom}\cref{prop:v:decom:viii}
		that
		$x^*\in (v-v_D-Ax)\cap B(x-v)= (0-Ax)\cap B(x-v_D)$.
		We conclude that $x\in \widetilde{Z}$.
	\end{proof}
	\begin{remark}
\label{r:heartpainback}
		Some comments on \cref{prop:Zvv:2shifts} are in order.
		\begin{enumerate}
			\item
			The assumption $v_R=0$ is critical in the conclusion of
			\cref{prop:Zvv:2shifts}\cref{prop:Zvv:2shifts:ii}
			as we illustrate in \cref{ex:not:eq:Z} below.
      \cref{ex:not:eq:Z} also shows that the inclusion \cref{prop:Zvv:2shifts}\cref{prop:Zvv:2shifts:i} cannot be improved to equality in general.
			\item
			The converse of \cref{prop:Zvv:2shifts}\cref{prop:Zvv:2shifts:ii}
			is not true
			as we illustrate in \cref{ex:eq:Z} below.
		\end{enumerate}
	\end{remark}

Before we present the limiting examples announced in 
\cref{r:heartpainback}, we recall that if 
	$C$ is a nonempty closed convex subset of $X$ and 
	$a\in X$, then 
	\begin{equation}
		\label{eq:shift:NC}
		\N{a+C}=\N{C}(\cdot-a).
	\end{equation}
	
	\begin{example}
		\label{ex:not:eq:Z}
		Suppose that $X=\RR^2$,
		let $\gamma<0$,
		and
		$(\alpha,\beta,\delta)\in \RR^3$.
		Set $a=(\alpha,\beta)$,
		$b=(\gamma,\delta)$,
		$K=\RR_+\times \{0\}$,
		$(A,B)=(\N{a+K},b+\N{K})$,
		and  $\widetilde{Z}= \tmenge{x\in X}{0\in -v_R+Ax+B(x-v_D)}$.
		Then the following hold:
		\begin{enumerate}
    \setlength\itemsep{0.35em}
			\item
			\label{ex:not:eq:Z:i}
			$\dom A-\dom B=
			\RR\times \{\beta\}$.
			\item
			\label{ex:not:eq:Z:ii}
			$\ran A+\ran B
			=\left]-\infty,\gamma\right]\times \RR$.
			\item
			\label{ex:not:eq:Z:iii:0}
			$v_D=(0,\beta)$.
			\item
			\label{ex:not:eq:Z:iii}
			$v_R=(\gamma,0)\neq(0,0)$.
			\item
			\label{ex:not:eq:Z:iv}
			$v=(\gamma,\beta)$.
			\item
			\label{ex:not:eq:Z:v}
			$Z=\left[\max\{\gamma,\alpha\},+\infty\right[ \times \{\beta\}$.
			\item
			\label{ex:not:eq:Z:v:0}
			$\widetilde{Z}
			=\left[\max\{0,\alpha\},+\infty\right[ \times \{\beta\}$.
			\item
			\label{ex:not:eq:Z:vi}
			$\widetilde{Z}
			\subsetneqq Z \siff \alpha<0$.
		\end{enumerate}
	\end{example}
	\begin{proof}
		\cref{ex:not:eq:Z:i}
		Indeed,
		$\dom A-\dom B
		=a+(K-K)
		=a+\RR\times \{0\}
		=\RR\times \{\beta\}$.
		
		\cref{ex:not:eq:Z:ii}:
		It follows from \cite[Theorem~3.1]{Zara}
		that
		$\ran A+\ran B
		=(\rec K)^\ominus+b+(\rec K)^\ominus
		=b+ K^\ominus
		=b+\RR_{-}\times \RR$.
		
		\cref{ex:not:eq:Z:iii:0}:
		It follows from \cref{ex:not:eq:Z:i}
		and \cref{e:def:vD:vR}
		that
		$v_D
		=\Pj{ \RR\times\{\beta \}  }(0,0)=(0,\beta)$.
		\cref{ex:not:eq:Z:iii}:
		It follows from \cref{ex:not:eq:Z:ii} and
		the assumption that $\gamma<0$
		that
		$v_R=\Pj{\left]-\infty,\gamma\right]\times \RR}(0,0)=(\gamma,0)$.
		
		\cref{ex:not:eq:Z:iv}:
		Combine \cref{ex:not:eq:Z:iii:0}, \cref{ex:not:eq:Z:iii} and \cref{prop:v:decom}\cref{prop:v:decom:viii}.
		
		\cref{ex:not:eq:Z:v}:
		Indeed, let $x\in \RR^2$. Then \cref{ex:not:eq:Z:iv}
		and \cref{eq:shift:NC} applied with
		$C$ replaced by $K$ yield
		\begin{subequations}
			\label{eq:cone:ex}
			\begin{align}
				x\in Z & \siff (0,0)\in (-\gamma,-\beta)+\N{(\alpha,\beta)+K}x+(\gamma,\delta)
				+\N{K}(x-(\gamma,\beta))
				\\
				& \siff (0,0)\in (0,\delta-\beta)+\N{K}(x-(\alpha,\beta))
				+\N{K}(x-(\gamma,\beta)).
			\end{align}
		\end{subequations}
		Set $Y=\tmenge{(x_1,\beta)\in \RR^2}{x_1\ge\max\{\gamma,\alpha\}}$.
		On the one hand,
		in view of \cref{eq:cone:ex},
		we learn that
		$(\forall (x_1,x_2)\in Z)$ we must have
		$x_1\ge \max\{\gamma,\alpha\}$
		and $x_2=\beta$. Hence, $Z\subseteq Y$.
		On the other hand,
		$(\forall x=(x_1,x_2)\in Y)$
		we have
    \begin{subequations}
		\begin{align}
			(0,0)&\in (0,\delta-\beta)
			+\{0\}\times \RR+\{0\}\times \RR \\ 
      &\subseteq
			(0,\delta-\beta)+\N{K}(x-(\alpha,\beta))
			+\N{K}(x-(\gamma,\beta)).
		\end{align}
    \end{subequations}
		Hence \cref{eq:cone:ex} implies that $Y\subseteq Z$.
		Altogether, we conclude that
		\cref{ex:not:eq:Z:v} holds.
		
		\cref{ex:not:eq:Z:v:0}:
		Let $x\in \RR^2$. Then
		\cref{eq:shift:NC} applied with
		$C$ replaced by $K$, \cref{ex:not:eq:Z:iii}
		and \cref{ex:not:eq:Z:iii:0} yield
		\begin{subequations}
			\label{eq:cone:ex:tilde}
			\begin{align}
				x\in \widetilde{Z}
				& \siff (0,0)\in (-\gamma,0)+\N{(\alpha,\beta )+K}x+(\gamma,\delta)
				+\N{K}(x-(0,\beta))
				\\
				& \siff (0,0)\in (0,\delta )+\N{K}(x-(\alpha,\beta))
				+\N{K}(x-(0,\beta)).
			\end{align}
		\end{subequations}
		Set $\widetilde{Y}=\tmenge{(x_1,\beta)\in \RR^2}{x_1\ge\max\{0,\alpha\}}$.
		On the one hand, in view of \cref{eq:cone:ex:tilde}
		$(\forall (x_1,x_2)\in \widetilde{Z})$
		we must have
		$x_1\ge \max\{0,\alpha\}$
		and $x_2=\beta$.
		Hence $\widetilde{Z}\subseteq \widetilde{Y}$.
		On the other hand,
		$(\forall x=(x_1,x_2)\in \widetilde{Y})$ 
		we have
		\begin{equation}
			(0,0)\in (0,\delta)
			+\{0\}\times \RR+\{0\}\times \RR \subseteq
			(0,\delta)+\N{K}(x-(\alpha,\beta))
			+\N{K}(x-(0,\beta))
			.
		\end{equation}
		Hence \cref{eq:cone:ex:tilde} yields $\widetilde{Y}\subseteq \widetilde{Z}$.
		Altogether, we conclude that
		\cref{ex:not:eq:Z:v:0} holds.
		
		\cref{ex:not:eq:Z:vi}:
		Indeed,
		\cref{ex:not:eq:Z:v}
		and   \cref{ex:not:eq:Z:v:0} imply that
		$ \widetilde{Z}=Z$ $\siff$ $\max\{\gamma,\alpha\}=\max\{0,\alpha\}$
		$\siff \alpha\ge 0$.
	\end{proof}

	\begin{example}
		\label{ex:eq:Z}
		Let $U$ be a closed linear subspace of $X$
		and let $(a,b )\in U^\perp\times U$.
		Suppose that $(A,B)=(\N{a+U},b+\N{U})$ and set
		$\widetilde{Z}\coloneqq \tmenge{x\in X}{0\in -v_R+Ax+B(x-v_D)}$.
		Then the following hold:
		\begin{enumerate}
    \setlength\itemsep{0.35em}
			\item
			\label{ex:eq:Z:i}
			$\dom A-\dom B=a+U$.
			\item
			\label{ex:eq:Z:ii}
			$\ran A+\ran B=b+U^\perp$.
			\item
			\label{ex:eq:Z:iii}
			$(v_D,v_R)=(a,b)\in U^\perp\times U$.
			\item
			\label{ex:eq:Z:iv}
			$v=a+b$.
			\item
			\label{ex:eq:Z:v}
			$Z=a+U$.
			\item
			\label{ex:eq:Z:vi}
			$\widetilde{Z}= Z$.
		\end{enumerate}
	\end{example}
	\begin{proof}
		\cref{ex:eq:Z:i}\&\cref{ex:eq:Z:ii}:
		This is clear.
		
		\cref{ex:eq:Z:iii}:
		It follows from \cref{ex:not:eq:Z:i}, \cref{ex:eq:Z:ii}
		and, e.g.,
		\cite[Proposition~3.19~and~Corollary~3.24(iii)]{BC}
		that
		$v_D=\Pj{a+U}(0)=a-\Pj{U}a=\Pj{U^\perp }a=a$,
		and
		$v_R=\Pj{b+U^\perp}(0)=b-\Pj{U^\perp}b=\Pj{U }b=b$.
		
		\cref{ex:eq:Z:iv}:
		Combine \cref{ex:eq:Z:iii} and \cref{prop:v:decom}\cref{prop:v:decom:viii}.
		\cref{ex:eq:Z:v}:
		Indeed, let $x\in X$. Then \cref{ex:eq:Z:iv},
		\cref{eq:shift:NC},
		and the assumption that
		$(a,b)\in U^\perp\times U$ imply
		that
		\begin{subequations}
			\begin{align}
				x\in Z & \siff   0\in -a-b+\N{a+U}(x)+b+\N{U}(x-a-b)
				\\
				& \siff   a\in \N{U}(x-a)+\N{U}(x-a-b)
				\\
				& \siff a\in U^\perp , x-a\in U,\text{ and }x-a-b\in U
				\\
				& \siff {x\in a+U}.
			\end{align}
		\end{subequations}
		\cref{ex:eq:Z:vi}:
		Let $x\in Z=a+U$ by \cref{ex:eq:Z:v}.
		Then $-v_R+Ax+B(x-v_D)=-b+\N{a+U}(x)+b+\N{U}(x-a)
		=U^\perp+U^\perp=U^\perp\ni 0$.
		Hence, $Z\subseteq \widetilde{Z}$.
		The opposite inclusion follows from
		\cref{prop:Zvv:2shifts}\cref{prop:Zvv:2shifts:i}.
	\end{proof}

	\section{Dynamic consequences}

\label{sec:dynamic}

In this section, we analyze the shadow
sequence $(\J{A}T^nx)_\nnn$ with regards to 
Fej\'er monotonicity and conditions necessary for convergence.
	Recall that if $x\in X$,
	then
	\begin{equation}
		\label{e:T:asym:reg}
		\J{A}T^n x-\J{B}\R{A} T^n x
		=\J{A^{-1}}T^n x+\J{B^{-1}}\R{A} T^n x
		=T^n x-T^{n+1}x\to v,
	\end{equation}
	where the identities are consequences
	of \cref{e:Id-T:desc} and
	the limit follows from, e.g., \cite[Corollary~1.5]{Br-Reich77}.

	\begin{proposition}
		\label{prop:pd:shadows}
		Let $x\in X$.
		Then the following hold:
		\begin{enumerate}
			\setlength\itemsep{0.35em}
			\item
			\label{prop:pd:shadows:ii}
			$\norm{\J{A}T^{n}x-\J{A}T^{n+1}x}^2
			+\norm{\J{A^{-1}}T^{n}x-\J{A^{-1}}T^{n+1}x}^2
			\to \norm{v}^2=\norm{v_D}^2+\norm{v_R}^2$.
			\item
			\label{prop:pd:shadows:iii}
			$\J{A}T^{n}x-\J{A}T^{n+1}x\to v_R$.
			\item
			\label{prop:pd:shadows:iv}
			$\J{A^{-1}}T^{n}x-\J{A^{-1}}T^{n+1}x\to v_D$.
		\end{enumerate}
	\end{proposition}
	
	\begin{proof}
		\cref{prop:pd:shadows:ii}:
		Indeed, we have
		\begin{subequations}
			\label{prop:pd:shadows:i}
			\begin{align}
				& \hspace{-2 cm}\norm{T^{n}x-T^{n+1}x}^2-\norm{\R{A}T^{n}x-\R{A}T^{n+1}x}^2
				\nonumber
				\\
				& =\norm{\J{A}T^n x-\J{A}T^{n+1} x}^2
				+\norm{\J{A^{-1}}T^n x-\J{A^{-1}}T^{n+1} x}^2 \\
				& \qquad +2\sscal{\J{A}T^n x-\J{A}T^{n+1} x}{\J{A^{-1}}T^n x-\J{A^{-1}}T^{n+1} x}
				\notag \\
				& \quad-\Big(\norm{\J{A}T^n x-\J{A}T^{n+1} x}^2
				+\norm{\J{A^{-1}}T^n x-\J{A^{-1}}T^{n+1} x}^2\nonumber                         \\
				& \qquad
				-2\sscal{\J{A}T^n x-\J{A}T^{n+1} x}{\J{A^{-1}}T^n x-\J{A^{-1}}T^{n+1} x}\Big)
				\nonumber \\
				& =4\sscal{\J{A}T^n x-\J{A}T^{n+1} x}{\J{A^{-1}}T^n x-\J{A^{-1}}T^{n+1} x}      \\
				& \to 0, \label{e:210620a} 
			\end{align}
		\end{subequations}
		where \cref{e:210620a} follows from \cite[Corollary~4.2]{BM:MPA17}.
		\cref{prop:pd:shadows:i} and \cref{e:T:asym:reg} imply that
		\begin{equation}
			\label{prop:pd:shadows:i:ii}
			\norm{\R{A}T^{n}x-\R{A}T^{n+1}x}
			\to \norm{v}.
		\end{equation}
		Next,
		\begin{subequations}
			\begin{align}
				& \hspace{-2 cm}\norm{\J{A}T^{n}x-\J{A}T^{n+1}x}^2
				+\norm{\J{A^{-1}}T^{n}x-\J{A^{-1}}T^{n+1}x}^2
				\nonumber
				\\
				& =\tfrac{1}{2}\Big(\norm{\J{A}T^{n}x-\J{A}T^{n+1}x}^2
				+\norm{\J{A}{A^{-1}}T^{n}x-\J{A^{-1}}T^{n+1}x}^2
				\\
				&
				\qquad
				+2\sscal{\J{A}T^n x-\J{A}T^{n+1} x}{
					\J{A^{-1}}T^n x-\J{A^{-1}}T^{n+1} x}
				\nonumber
				\\
				&
				\qquad +
				\norm{\J{A}T^{n}x-\J{A}T^{n+1}x}^2
				+\norm{\J{A^{-1}}T^{n}x-\J{A^{-1}}T^{n+1}x}^2
				\nonumber
				\\
				&
				\qquad
				-2\sscal{\J{A}T^n x-\J{A}T^{n+1} x}{\J{A^{-1}}T^n x-\J{A^{-1}}T^{n+1} x}
				\Big)
				\nonumber\\
				& =\tfrac{1}{2}\Big(\norm{\J{A}T^{n}x-\J{A}T^{n+1}x
					+\J{A^{-1}}T^{n}x-\J{A^{-1}}T^{n+1}x}^2
				\\
				&
				\qquad
				+\norm{\J{A}T^{n}x-\J{A}T^{n+1}x
					-(\J{A^{-1}}T^{n}x-\J{A^{-1}}T^{n+1}x)}^2\Big)
				\nonumber\\
				& =\tfrac{1}{2}\big(\norm{T^nx-T^{n+1}x}^2
				+\norm{\R{A}T^nx-\R{A}T^{n+1}x}^2\big)
				\\
				& \to \tfrac{1}{2}\big(\norm{v}^2+\norm{v}^2\big)=\norm{v}^2,
        \label{e:210620b}
			\end{align}
		\end{subequations}
    where \cref{e:210620b} follows from
		\cref{e:T:asym:reg}, \cref{prop:pd:shadows:i:ii}, and
		\cref{prop:v:decom}\cref{prop:v:decom:ix}.
		
		\cref{prop:pd:shadows:iii}\&\cref{prop:pd:shadows:iv}:
		In view of \cref{prop:v:decom}\cref{prop:v:decom:ix}, we
		rewrite \cref{prop:pd:shadows:ii}
		as
		\begin{equation}
			\norm{(\J{A}T^{n}x-\J{A}T^{n+1}x
				,\J{A^{-1}}T^{n}x-\J{A^{-1}}T^{n+1}x)}^2
			\to \norm{(v_R,v_D)}^2.
		\end{equation}
		\cref{lem:asymp:shadow}\cref{lem:asymp:shadow:i}\&\cref{lem:asymp:shadow:ii} implies that 
		the sequence $(\J{A}T^{n}x-\J{A}T^{n+1}x
		,\J{A^{-1}}T^{n}x-\J{A^{-1}}T^{n+1}x)_\nnn$
		lies in $(\ran A+\ran A)\times (\dom A-\dom B)
		\subseteq \overline{(\ran A+\ran A)}\times
		\overline{(\dom A-\dom B)} $.
		Using, e.g., \cite[Lemma~2]{Pazy1970}
		in view of \cref{e:proj:ps}
		we learn that
		$(\J{A}T^{n}x-\J{A}T^{n+1}x
		,\J{A^{-1}}T^{n}x-\J{A^{-1}}T^{n+1}x)\to (v_R,v_D)$;
		equivalently,
		$\J{A}T^{n}x-\J{A}T^{n+1}x\to v_R$
		and
		$\J{A^{-1}}T^{n}x-\J{A^{-1}}T^{n+1}x\to v_D$.
	\end{proof}

	Recall that a sequence $(y_n)_\nnn$ in $X$ is called 
	\emph{asymptotically regular} if $y_n-y_{n+1}\to 0$; 
  clearly, this is a condition necessary for convergence of $(y_n)_\nnn$. 
	
	\begin{proposition}[{\bf shadow convergence: necessary condition}]
		\label{prop:necessary:cond}
		Let $ x\in X$. Then the following hold:
		\begin{enumerate}
    \setlength\itemsep{0.35em}
			\item
			\label{prop:necessary:cond:i}
			$(\J{A}T^n x)_\nnn$ is asymptotically regular  $\siff$ $v_R=0$.
			\item
			\label{prop:necessary:cond:ii}
			$(\J{A^{-1}}T^n x)_\nnn$ is asymptotically regular  $\siff$ $v_D=0$.
			
		\end{enumerate}
	\end{proposition}
	\begin{proof}
		\cref{prop:necessary:cond:i}\&\cref{prop:necessary:cond:ii}:
		This is a direct consequence of
		\cref{prop:pd:shadows}\cref{prop:pd:shadows:iii}\&\cref{prop:pd:shadows:iv}.
	\end{proof}

	\begin{proposition}
		\label{prop:pd:sh:bdd}
		Suppose that  $ f\in \fix(v+T)$,
		let $x\in X$, and let $\nnn$.
		Then the  following hold:
		\begin{enumerate}
    \setlength\itemsep{0.35em}
			\item
			\label{prop:pd:sh:bdd:i}
			$T^n f=f-nv$.
			\item
			\label{prop:pd:sh:bdd:ii}
			$\J{A}T^n f=\J{A}f-nv_R$.
			\item
			\label{prop:pd:sh:bdd:iii}
			$\J{A^{-1}}T^n f=\J{A^{-1}}f-nv_D$.
			\item
			\label{prop:pd:sh:bdd:ii:b}
			$\J{B}\R{A}T^n f=\J{B}\R{A}f-nv_R$.
			\item
			\label{prop:pd:sh:bdd:iii:b}
			$\J{B^{-1}}\R{A}T^n f=\J{B^{-1}}\R{A}f+nv_D$.
		\end{enumerate}
	\end{proposition}
	
	\begin{proof}
		\cref{prop:pd:sh:bdd:i}:
		See \cite[Proposition~2.5(iv)]{101}.
		
		\cref{prop:pd:sh:bdd:ii}\&\cref{prop:pd:sh:bdd:iii}:
		We proceed by induction.
		When $n=0$ the statement is trivial and for
		$n=1$
		use  \cref{e:JA:Tf:f} and \cref{e:JAinv:Tf:f}.
		Now suppose that for some $n\ge 1$ it holds that
		$(\J{A}T^n f,\J{A^{-1}}T^n f)=(\J{A}f-nv_R,\J{A^{-1}}f-nv_D)$.
		\cref{lem:asymp:shadow}\cref{lem:asymp:shadow:i}\&\cref{lem:asymp:shadow:ii}
		applied with $x$ replaced by $T^nf$ yields that
		$\J{A}T^nf-\J{A}T^{n+1}f\in \ran A+\ran B$
		and
		$\J{A^{-1}}T^nf-\J{A^{-1}}T^{n+1}f\in \dom A-\dom B$.
		Now
		\begin{subequations}
			\begin{align}
				\norm{(v_R,v_D)}^2 & \le\norm{(\J{A}T^nf-\J{A}T^{n+1}f,\J{A^{-1}}T^nf-\J{A^{-1}}T^{n+1}f)}^2
				\\
				& =\norm{\J{A}T^nf-\J{A}T^{n+1}f}^2
				+\norm{\J{A^{-1}}T^nf-\J{A^{-1}}T^{n+1}f)}^2
				\\
				& \le \norm{\J{A}T^nf-\J{A}T^{n+1}f}^2
				+\norm{\J{A^{-1}}T^nf-\J{A^{-1}}T^{n+1}f)}^2
				\\
				& \qquad+2\sscal{\J{A}T^nf-\J{A}T^{n+1}f}{\J{A^{-1}}T^nf-\J{A^{-1}}T^{n+1}f}
				\nonumber\\
				& =\norm{T^nf-T^{n+1}f}^2
				=\norm{f-nv-(f-(n+1)v)}^2
				\\
				& = \norm{v}^2=\norm{(v_D,v_R)}^2.
			\end{align}
		\end{subequations}
		Thus, by definition of $v_D$ and $v_R$, 
		we have
		$(\J{A}T^nf-\J{A}T^{n+1}f,\J{A^{-1}}T^nf-\J{A^{-1}}T^{n+1}f)=(v_R,v_D)$.
		Recalling the inductive hypothesis,
		we learn that
    \begin{subequations}
		\begin{align}
			(\J{A}T^{n+1}f,\J{A^{-1}}T^{n+1}f)
			& =(\J{A}T^nf,\J{A^{-1}}T^nf)-(v_R,v_D)    \\
			& =(\J{A}f-(n+1)v_R,\J{A^{-1}}f-(n+1)v_D).
		\end{align}
    \end{subequations}

		\cref{prop:pd:sh:bdd:ii:b}:
		Using \cref{prop:pd:sh:bdd:i}
		we have $v= f-T f=(\Id -T) f=\J{A} f-\J{B}\R{A} f$,
		hence $\J{A}f=v+\J{B}\R{A}f$.
		Let $\nnn$. \cref{prop:pd:sh:bdd:i}
		and \cref{prop:pd:sh:bdd:ii}
		imply that
		$v=T^n f-T^{n+1}f
		=(\Id -T)T^n f
		=\J{A}T^n f-\J{B}\R{A} T^nf$.
		Hence,
		$\J{B}\R{A} T^nf
		=-v+\J{A}T^n f
		=-v+\J{A}f-nv_R
		=-v+v+\J{B}\R{A}f-nv_R
		=\J{B}\R{A}f-nv_R$.
		
		\cref{prop:pd:sh:bdd:iii:b}:
		Using \cref{prop:pd:sh:bdd:i}
		we have $v= f-T f=(\Id -T) f=\J{A^{-1}} f+\J{B^{-1}}\R{A} f$,
		hence $\J{A^{-1}}f=v-\J{B^{-1}}\R{A}f$.
		Let $\nnn$. \cref{prop:pd:sh:bdd:i}
		and \cref{prop:pd:sh:bdd:iii}
		imply that
		$v=T^n f-T^{n+1}f
		=(\Id -T)T^n f
		=\J{A^{-1}}T^n f+\J{B^{-1}}\R{A} T^nf$.
		Hence,
		$\J{B^{-1}}\R{A} T^nf
		=v-\J{A^{-1}}T^n f
		=v-(\J{A^{-1}}f-nv_D)
		=v-v+\J{B^{-1}}\R{A}f+nv_D
		=\J{B^{-1}}\R{A}f+nv_D$.
	\end{proof}

	We omit the simple proof of the following result.
	\begin{lemma}
		\label{lem:Fejer:shift}
		Let $s\in X$
		and let $C$ be a nonempty subset of $X$.
		Suppose that $(x_n)_\nnn$ is Fej\'er
		monotone with respect to $C$.
		Then $(s+x_n)_\nnn$ is Fej\'er
		monotone with respect to $s+C$.
	\end{lemma}
  
  We are now ready for the main result in this section.
  
	\begin{theorem}
		\label{prop:pd:sh:bdd:fejer}
		Suppose that  $ f\in \fix(v+T)$,
		let $x\in X$, and let $\nnn$.
		Then the  following hold:
		\begin{enumerate}
			\setlength\itemsep{0.35em}
			\item
			\label{prop:pd:sh:bdd:fejer:vi}
			The sequence $(\J{A}T^n x+nv_R,\J{A^{-1}}T^nx+nv_D )_\nnn$
			is Fej\'er monotone with respect to the set
			$\menge{(\J{A}\times \J{A^{-1}})(f,f)}{f\in \fix (v+T)}$.
			\item
			\label{prop:pd:sh:bdd:fejer:vii}
			The sequence $(\J{A}T^n x+nv_R,\J{A^{-1}}T^nx+nv_D )_\nnn$
			is bounded.
			\item
			\label{prop:pd:sh:bdd:fejer:vii:b}
			The sequence $(\J{B}\R{A}T^n x+nv_R,\J{B^{-1}}\R{A}T^nx-nv_D )_\nnn$
			is bounded.
			\item
			\label{prop:pd:sh:bdd:fejer:ix:x}
			The sequence $((0,-v)+(\J{A}T^n x+nv_R,\J{A^{-1}}T^nx+nv_D ))_\nnn$
			is Fej\'er monotone with respect to the set
			$\SE$.

			\item
			\label{prop:pd:sh:bdd:fejer:viii}
			The sequence $(\J{A}T^n x )_\nnn$
			is bounded $\siff$ $v_R=0$.
			\item
			\label{prop:pd:sh:bdd:fejer:xi}
			$v_R\neq 0$ $\RA$ $\norm{\J{A}T^n x }\to +\infty$.
			
			\item
			\label{prop:pd:sh:bdd:fejer:ix}
			The sequence $(\J{A^{-1}}T^nx )_\nnn$
			is bounded $\siff$ $v_D=0$.
			
			\item
			\label{prop:pd:sh:bdd:fejer:xii}
			$v_D\neq 0$ $\RA$ $\norm{\J{A^{-1}}T^n x }\to +\infty$.
			\item
			\label{prop:pd:sh:bdd:fejer:x}
			Suppose that $A$ and $B$ are paramonotone.
			Then
			the sequence $((0,-v)+(\J{A}T^n x+nv_R,\J{A^{-1}}T^nx+nv_D ))_\nnn$
			is Fej\'er monotone with respect to
			$Z\times K$.
		\end{enumerate}
	\end{theorem}
	
	\begin{proof}
		\cref{prop:pd:sh:bdd:fejer:vi}\&\cref{prop:pd:sh:bdd:fejer:vii}:
		Indeed, it follows from
		\cite[Theorem~2.7(v)]{BM:MPA17}
		that $(\forall (x,y)\in X\times X)$
		\begin{align}
			&\hspace{-2cm} \norm{\J{A} T^{n+1} x-\J{A} T^{n+1} y}^2+
			\norm{\J{A^{-1}} T^{n+1} x-\J{A^{-1}} T^{n+1} y}^2
      \nonumber\\
			&\le   \norm{\J{A} T^{n} x-\J{A} T^{n} y}^2+
			\norm{\J{A}{A^{-1}} T^{n} x-\J{A^{-1}} T^{n} y}^2.
		\end{align}
		Applying the above equation with $y$ replaced by
		$f\in \fix (v+T)$
    and recalling
		\cref{prop:pd:sh:bdd}\cref{prop:pd:sh:bdd:ii}\&\cref{prop:pd:sh:bdd:iii}
		yields
		\begin{align}
			& \hspace{-2cm} \norm{\J{A} T^{n+1} x-(\J{A}  f-(n+1)v_R)}^2+
			\norm{\J{A^{-1}} T^{n+1} x-(\J{A^{-1}} f-(n+1)v_D)}^2
			\nonumber
			\\
			& \le   \norm{\J{A} T^{n} x-(\J{A}f-nv_R)}^2+
			\norm{\J{A^{-1}} T^{n} x-(\J{A^{-1}} f-nv_D)}^2.
		\end{align}
		Rearranging yields
		\begin{align}
			& \hspace{-2cm}
			\norm{\J{A} T^{n+1} x+(n+1)v_R-\J{A}  f}^2+
			\norm{\J{A^{-1}} T^{n+1 } x+(n+1)v_D-\J{A^{-1}}f}^2
			\nonumber
			\\
			& \le   \norm{\J{A} T^{n} x+nv_R-\J{A}f}^2+
			\norm{\J{A^{-1}} T^{n} x+nv_D-\J{A^{-1}} f}^2,
		\end{align}
		and the conclusion follows.

		\cref{prop:pd:sh:bdd:fejer:vii:b}:
		Combine \cref{prop:pd:sh:bdd:fejer:vii}
		and \cref{e:T:asym:reg}.
		
		\cref{prop:pd:sh:bdd:fejer:ix:x}:
		Combine \cref{prop:pd:sh:bdd:fejer:vi},
		\cref{lem:bad:walaa}\cref{lem:bad:walaa:iii}
		and \cref{lem:Fejer:shift} applied with
		$s$ replaced by $(0,-v)$
		and
		$C$ replaced by
		$\tmenge{(\J{A}\times \J{A^{-1}})(f,f)}{f\in \fix (v+T)} $.
		
		\cref{prop:pd:sh:bdd:fejer:viii}\&\cref{prop:pd:sh:bdd:fejer:xi}:
		``$\RA$": Indeed, observe that
		$\norm{\J{A}T^n x}\ge n\norm{v_R}-\norm{\J{A}T^n x+nv_R}$.
		``$\LA$": This is a direct consequence of
		\cref{prop:pd:sh:bdd:fejer:vii} applied with $v_R=0$.
		
		\cref{prop:pd:sh:bdd:fejer:ix}\&\cref{prop:pd:sh:bdd:fejer:xii}:
		``$\RA$": Indeed, observe that
		$\norm{\J{A^{-1}}T^n x}\ge n\norm{v_D}-\norm{\J{A^{-1}}T^n x+nv_D}$.
		``$\LA$": This is a direct consequence of
		\cref{prop:pd:sh:bdd:fejer:vii} applied with $v_D=0$.
    
		\cref{prop:pd:sh:bdd:fejer:x}:
		Combine
		\cref{prop:pd:sh:bdd:fejer:ix:x}
		and \cref{fact:para:cc}\cref{fact:para:ii:b}.
	\end{proof}

	\begin{remark}[{\bf unbounded shadows}]
		\label{prop:necessary:cond:iv}
		In view of \cref{prop:pd:sh:bdd:fejer}\cref{prop:pd:sh:bdd:fejer:vii}
		and
		\cref{prop:v:decom}\cref{prop:v:decom:ix},
		we learn that if $\fix T=\fady$ (in particular, if $v\neq 0$),
		then  at least one of the sequences
		$(\J{A}T^n x)_\nnn$, $(\J{A^{-1}}T^n x)_\nnn$
		is unbounded.
	\end{remark}

	\begin{corollary}[{\bf boundedness of the primal shadows}]
		\label{cor:p:sh:bdd}
		Suppose that $v_R=0$.  Suppose that  $ f\in \fix(v+T)$,
		and let $x\in X$.
		Then the  following hold:
		\begin{enumerate}
			\setlength\itemsep{0.35em}
			\item
			\label{cor:p:sh:bdd:i}
			$(\forall \nnn)$ $\J{A}T^n f=\J{A}(f-nv_D)=\J{A}f$.
			\item
			\label{cor:p:sh:bdd:ii}
			$(\J{A} T^n x)_\nnn $ is bounded.
			\item
			\label{cor:p:sh:bdd:iii}
			$(\J{B}\R{A} T^n x)_\nnn $ is bounded.
		\end{enumerate}
	\end{corollary}
	
	\begin{proof}
		\cref{cor:p:sh:bdd:i}:
		Combine
		\cref{prop:v:decom}\cref{prop:v:decom:viii} and
		\cref{prop:pd:sh:bdd}\cref{prop:pd:sh:bdd:i}\&\cref{prop:pd:sh:bdd:ii} with $v_R=0$.
		\cref{cor:p:sh:bdd:ii}:
		Apply \cref{prop:pd:sh:bdd:fejer}\cref{prop:pd:sh:bdd:fejer:vii} with $v_R=0$.
		\cref{cor:p:sh:bdd:iii}:
		Combine \cref{cor:p:sh:bdd:ii}
		and \cref{e:T:asym:reg}.
	\end{proof}
	\begin{corollary}[{\bf boundedness of the dual shadows}]
		\label{cor:d:sh:bdd}
		Suppose that $v_D=0$.  Suppose that  $ f\in \fix(v+T)$,
		and let $x\in X$.
		Then the  following hold:
		\begin{enumerate}
			\setlength\itemsep{0.35em}
			\item
			\label{cor:d:sh:bdd:i}
			$(\forall \nnn)$ $\J{A^{-1}}T^n f=\J{A^{-1}}(f-nv_R)=\J{A^{-1}}f$.
			\item
			\label{cor:d:sh:bdd:ii}
			$(\J{A^{-1}} T^n x)_\nnn $ is bounded.
			\item
			\label{cor:d:sh:bdd:iii}
			$(\J{B^{-1}}\R{A} T^n x)_\nnn $ is bounded.
		\end{enumerate}
	\end{corollary}
	
	\begin{proof}
		\cref{cor:p:sh:bdd:i}:
		Combine
		\cref{prop:v:decom}\cref{prop:v:decom:viii} and
		\cref{prop:pd:sh:bdd}\cref{prop:pd:sh:bdd:i}\&\cref{prop:pd:sh:bdd:iii} with $v_D=0$.
		\cref{cor:p:sh:bdd:ii}:
		Apply \cref{prop:pd:sh:bdd:fejer}\cref{prop:pd:sh:bdd:fejer:vii} with $v_D=0$.
		\cref{cor:d:sh:bdd:iii}:
		Combine \cref{cor:d:sh:bdd:ii}
		and \cref{e:T:asym:reg}.
	\end{proof}
	
	\begin{corollary}
		\label{cor:Pazy:split}
		Let $x\in X$. Then the following hold:
		\begin{enumerate}
    \setlength\itemsep{0.35em}
			\item
			\label{cor:Pazy:split:i}
			$\frac{1}{n}(\J{A}T^n x)\to -v_R$.
			\item
			\label{cor:Pazy:split:ii}
			$\frac{1}{n}(\J{A^{-1}}T^n x)\to -v_D$.
		\end{enumerate}
	\end{corollary}
	\begin{proof}
		This is a direct consequence
		of \cref{prop:pd:sh:bdd:fejer}\cref{prop:pd:sh:bdd:fejer:vii}.
	\end{proof}

	\section{Shadow convergence!}
  
  \label{sec:yay}

	In this section, we shall prove the main result announced in 
  \cref{sec:intromain}. We assume throughout that
	\begin{empheq}[box=\mybluebox]{equation*}
		\label{e:f:g}
		\text{$f\in \Gamma_0(X)$\;
			and\; $g\in \Gamma_0(X)$,}
	\end{empheq}
	that $(A,B)=(\partial f,\partial g)$,
	and 
	(see \cref{prop:v:decom}\cref{prop:v:decom:viii}) that 
	\begin{empheq}[box=\mybluebox]{equation}
		\label{e:v=vd}
		v_R=0\quad \Leftrightarrow \quad v=\Pj{\cran  (\Id-T)}(0)=\Pj{\overline{\dom f-\dom g}}(0)=v_D.
	\end{empheq}
We remind the reader on our abbreviations
	\begin{empheq}[box=\mybluebox]{equation*}
		\label{e:f:g:prox}
		\big(\prox_f,\prox_{f^*},\prox_g, \refl_f\big)
		=\big(\operatorname{Prox}_f, \operatorname{Prox}_{f^*},\operatorname{Prox}_g, 2\operatorname{Prox}_f-\Id\big).
	\end{empheq}
Then 
	\begin{equation}
		\label{e:def:T:fg}
		T= T_{(\partial f,\partial g)} = \Id-\prox_f+\prox_g R_f.
	\end{equation}

	\begin{remark}
		Let $x\in X$. In view of \cref{prop:necessary:cond}\cref{prop:necessary:cond:i}, 
		applied with $(A,B)$ replaced by $(\partial f,\partial g)$,
		we learn that the assumption $v_R=0$ is \emph{necessary} for the convergence of the
		shadow sequence $(\prox_fT^n x)_\nnn$.
	\end{remark}

	Because
	$v=v_D$, the definition of
	$v_D$ implies 
	\begin{equation}
		\label{e:v:loc}
(\forall (a,b)\in \cdom f\times \cdom g) \quad 
		\scal{v}{v-(a-b)}\le 0.
	\end{equation}
The following result provides several descriptions of $Z$ in the 
optimization setting:

	\begin{proposition}
		\label{prop:fg:zv}
		Recalling \cref{e:def:Zv}, we have:
		\begin{enumerate}
			\setlength\itemsep{0.35em}
			\item
			\label{prop:fg:zv:i}
			\    $Z=\menge{x\in X}{0\in \partial f(x)+\partial g(x-v)}$.
			\item
			\label{prop:fg:zv:ii}
			$Z\neq \fady$ $\RA$
			$Z=\displaystyle\argmin_{x\in X} (-\scal{x}{v}+f(x)+g(x-v))=\argmin_{x\in X} (f(x)+g(x-v)).$
		\end{enumerate}
	\end{proposition}
	\begin{proof}
		\cref{prop:fg:zv:i}:
		Apply \cref{prop:Zvv:2shifts}\cref{prop:Zvv:2shifts:ii}, 
		with $(A,B)$
		replaced by
		$(\partial f, \partial g)$, and use \cref{e:v=vd}.
		\cref{prop:fg:zv:ii}:
		Combine \cref{e:def:Zv}, \cref{prop:fg:zv:i} and \cref{prop:criticals:mini}
		applied twice with $(f,g)$ replaced by $(-\scal{v}{\cdot}+f,g(\cdot-v))$
		and by $(f,g(\cdot-v))$ respectively.
	\end{proof}
	
	The next result is a key step towards our main result:

	\begin{lemma}[{\bf the prox lemma}]
		\label{lem:fg:ineq}
		Let $x\in X$ and let $y\in \dom f\cap (v+\dom g)$.
		Then the following hold:
		\begin{subequations}
			\begin{align}
				0      & \ge \scal{y-\prox_f x}{v}  ,
				\label{prop:fg:ineq:i}
				\\
				f(y)   & \ge f(\prox_f x)+\scal{y-\prox_f x}{v+\prox_{f^*}x},
				\label{prop:fg:ineq:ii}
				\\
				g(y-v) & \ge g(\prox_g\refl_f  x)
				+\scal{y-\prox_g\refl_f  x-v}{\prox_f x-\prox_g\refl_f   x-v}
				\label{prop:fg:ineq:iii}
				\\
				& \quad
				+\scal{-y+v+\prox_g\refl_f  x}{v+\prox_{f^*}x}.
				\nonumber
			\end{align}
		\end{subequations}
		
	\end{lemma}
	\begin{proof}
		Applying  \cref{e:v:loc}
		with $(a,b)$ replaced by
		$(\prox_f x,y-v)\in \dom \partial f\times \dom g$
		yields
		$\scal{\prox_f x-(y-v)-v}{0-v} \le 0$; 
		equivalently, $\scal{y-\prox_f x}{v} \le 0$, which 
		is \cref{prop:fg:ineq:i}.
		We now prove \cref{prop:fg:ineq:ii}.
		Indeed, the characterization of $\prox_f$
		and \cref{prop:fg:ineq:i}
		yield
		\begin{subequations}
			\label{e:prox:f}
			\begin{align}
				f(y) & \ge f(\prox_f x)+\scal{y-\prox_f x}{x-\prox_f x}
				\\
				& \ge    f(\prox_f x)+\scal{y-\prox_f x}{\prox_{f^*} x}
				+\scal{y-\prox_f x}{v}
				\\
				& =f(\prox_f x)+\scal{y-\prox_f x}{v+\prox_{f^*} x}.
			\end{align}
		\end{subequations}
		Finally, we turn to \cref{prop:fg:ineq:iii}.
		Indeed, the characterization of $\prox_g$
		and \cref{e:v:loc}
		applied with $(a,b)$
		replaced by $(y, \prox_g\refl_f x)$
		in \cref{e:proj:used}
		yields 
		\begin{subequations}
			\label{e:prox:g}
			\begin{align}
				g(y-v)
				& \ge g(\prox_g \refl_f x)
				+\scal{y-v-\prox_g \refl_f x}{\refl_f x-\prox_g \refl_f x}
				\\
				& =
				g(\prox_g \refl_f x)
				+2\scal{y-\prox_g \refl_f x-v}{v}
				\\
				& \quad +\scal{y-\prox_g \refl_f x-v}{\prox_f x-\prox_g\refl_f x-v}
				\nonumber
				\\
				& \quad -\scal{y-\prox_g \refl_f x-v}{v+\prox_{f^*} x  }
				\nonumber\\
				& \ge   g(\prox_g \refl_f x)+\scal{y-\prox_g \refl_f x-v}{\prox_f x-\prox_g\refl_f x-v} \label{e:proj:used}
				\\
				& \quad
				-\scal{y-\prox_g \refl_fTx-v}{v+\prox_{f^*} x}.
				\nonumber 
			\end{align}
		\end{subequations}
		The proof is complete.
	\end{proof}

	To make further progress, we recall that $v=v_D$ and
	we assume additionally from now on 
	(see \cref{e:Z:Fix}
	and \cref{lem:pd:sh:bdd:iii:c}\cref{prop:pd:sh:bdd:v}) that
	\begin{empheq}[box=\mybluebox]{equation}
		\label{e:v:attained}
		Z\neq \fady;\;\text{thus, }\; v\in \ran (\Id-T)\cap(
		\dom \partial f-\dom \partial g).
	\end{empheq}

	We recall that \cref{e:T:asym:reg}, 
	applied with $(A,B)$ replaced by
	$(\partial f, \partial g)$, and 
	of \cite[Example~23.3]{BC} imply 
	\begin{equation}
  (\forall x\in X)\quad 
		\label{e:seq:strong:lim}
		\prox_fT^n x-\prox_g\refl_f T^n x=
		T^n x-T^{n+1}x\to v.
	\end{equation}
	
	\begin{proposition}
		\label{prop:fg:ineq}
		Let $y\in \dom f\cap (v+\dom g)$
		and let $x\in X$. Set $(\forall \nnn)$
		\begin{subequations}
			\label{e:def:eps:alpha:n}
			\begin{align}
				\epsilon_n & =\scal{y-\prox_g \refl_fT^n x-v}{\prox_fT^n x-\prox_g\refl_fT^n x-v}, 
				\\
				\delta_n   & =\scal{\prox_fT^n x-\prox_g\refl_f T^n  x-v}{\prox_fT^n x-(T^n x+nv)}.
			\end{align}
		\end{subequations}
		Then
		\begin{equation}
			\label{eq:lim:eps}
			\epsilon_n\to 0\qquad\text{and}\qquad \delta_n\to 0.
		\end{equation}
		Moreover, 
		\begin{equation}
    (\forall n\ge 1)\quad 
			\label{prop:fg:ineq:iv}
			f(y)+g(y-v)\ge f(\prox_fT^n x)+g(\prox_g \refl_fT^n x)+\epsilon_n+\delta_n.
		\end{equation}
	\end{proposition}
  
	\begin{proof}
		Recall
		that $(T^nx+nv)_\nnn$ is a bounded sequence by
		\cite[Proposition~2.5(vi)]{101},
		that $(\prox_f T^n x)_\nnn$ and $(\prox_g\refl_fT^n x)_\nnn$
		are bounded sequences by
		\cref{cor:p:sh:bdd}\cref{cor:p:sh:bdd:ii}\&\cref{cor:p:sh:bdd:iii}
		applied with $(A,B)=(\partial f, \partial g)$.
		Combining this with \cref{e:seq:strong:lim}
		proves \cref{eq:lim:eps}.
		
		We now turn to   \cref{prop:fg:ineq:iv}.
		Adding \cref{prop:fg:ineq:ii} and \cref{prop:fg:ineq:iii}
		applied with $x$ replaced by $T^n x$
		yields
		\begin{subequations}
			\label{e:prox:fg}
			\begin{align}
				f(y)+g(y-v)
				& \ge
				f(\prox_fT^n x)+\scal{y-\prox_fT^n x}{v+\prox_{f^*}T^n x}+g(\prox_g \refl_fT^n x)
				\\
				& \quad
				+\epsilon_n -\scal{y-\prox_g \refl_fT^n x-v}{v+\prox_{f^*}T^n x}
				\nonumber
				\\
				& = f(\prox_fT^n x)+g(\prox_g \refl_fT^n x)+\epsilon_n
				\\
				& \quad
				+\scal{v+\prox_g \refl_fT^n x-\prox_fT^n x}{v+\prox_{f^*}T^n x}
				\nonumber\\
				& = f(\prox_fT^n x)+g(\prox_g \refl_fT^n x)+\epsilon_n
				\\
				& \quad
				+\scal{v+\prox_g \refl_fT^n x-\prox_fT^n x}{v+T^n x+nv-\prox_{f}T^n x-nv}
				\nonumber\\
				& = f(\prox_fT^n x)+g(\prox_g \refl_fT^n x)+\epsilon_n+\delta_n
				\\
				& \quad
				+(1-n)\scal{v-(\prox_fT^n x-\prox_g \refl_fT^n x)}{v-0}
				\nonumber\\
				& \ge  f(\prox_fT^n x)+g(\prox_g \refl_fT^n x)+\epsilon_n+\delta_n,
				\label{e:drop:pr}
			\end{align}
		\end{subequations}
		where \cref{e:drop:pr} follows from applying
		\cref{e:v:loc} with $(a,b)$ replaced by
		$(\prox_fT^n x,\prox_g \refl_fT^n x)\in \dom \partial f\times \dom \partial g$.
	\end{proof}

	\begin{proposition}
		\label{prop:clusters}
    Set $\mu\coloneqq\min_{x\in X}(f(x)+g(x-v))$ and
		let $x\in X$. Then the following hold:
		\begin{enumerate}
    \setlength\itemsep{0.35em}
			\item
			\label{prop:clusters:i}
      $(\prox_fT^n x)_\nnn$ is bounded
			and its weak  cluster points are minimizers of
			$f+g(\cdot-v)$.
			\item
			\label{prop:clusters:ii}
			$(\prox_g\refl_fT^n x)_\nnn$ is bounded
			and its weak  cluster points are minimizers of
			$f(\cdot+v)+g$.
      \end{enumerate}
      Now let $\overline{z}$ be a weak cluster point of $(\prox_fT^n x)_\nnn$. 
      Then: 
      \begin{enumerate}
    \setlength\itemsep{0.35em}
    \setcounter{enumi}{2}
			\item
			\label{prop:clusters:iv}
			$f(\prox_fT^n x)\to f(\overline{z})$.
			\item
			\label{prop:clusters:v}
			$g(\prox_g\refl_fT^n x)\to g(\overline{z}-v)$.
			\item
			\label{prop:clusters:iii}
			$f(\prox_fT^n x)+g(\prox_g\refl_fT^n x)\to \mu$.
		\end{enumerate}
	\end{proposition}

	\begin{proof}
		Set $(\forall n\ge 1)$ $(p_n,q_n)=(\prox_f T^n x, \prox_g\refl_fT^n x)$.
    
		\cref{prop:clusters:i}:
		\cref{cor:p:sh:bdd}\cref{cor:p:sh:bdd:ii}\&\cref{cor:p:sh:bdd:iii}
		applied with $(A,B)$
		replaced by $(\partial f, \partial g)$
		imply that
		$(p_n)_\nnn$
		and    $(q_n)_\nnn$
		are bounded sequences.
		Let $\overline{z}$ be a weak cluster point of
		$(p_n)_\nnn$
		and observe that by \cref{e:seq:strong:lim}
		$\overline{z}-v$ is a weak cluster point of
		$(q_n)_\nnn$.
		Let $y\in \dom f\cap\dom g(\cdot-v)$.
		The (weak) lower semicontinuity of $f$
		and $g$
		in view of \cref{prop:fg:ineq:iv}
		yields
		\begin{subequations}
			\label{e:limits}
			\begin{align}
				f(y)+g(y-v)
				&
				\ge \varlimsup (f(p_n)+g(q_n))
				\\
				&
				\ge \varliminf f(p_n)+ \varliminf g(q_n)
				\ge f(\overline{z})+g(\overline{z}-v).
			\end{align}
		\end{subequations}
		This implies that
		\begin{equation}
			\overline{z}\in \dom f\cap\dom g(\cdot-v).
		\end{equation}
		Observe that  \cref{prop:fg:ineq:iv}
		implies that
		$f(y)+g(y-v)\ge f(p_n)+g(q_n)+\epsilon_n+\delta_n$.
		Pick $(k_n)_\nnn$ such that
		$f(p_{k_n})\to \varlimsup f(p_n)$. Then
    \begin{subequations}
			\label{e:limits:pn}
		\begin{align}
			f(y)+g(y-v)&\ge \lim f(p_{k_n})+\varlimsup g(q_{k_n})
			\ge \varliminf f(p_n)+\varlimsup g(q_{k_n})\\
			&\ge f(\overline{z})+g(\overline{z}-v).
		\end{align}
    \end{subequations}
		Setting $y=\overline{z}$ in \cref{e:limits:pn} yields
		\begin{equation}
			\lim f(p_{k_n}) =\varlimsup f(p_n) = \varliminf f(p_n).
		\end{equation}
		Hence $(f(p_n))_\nnn$ converges.
		Similarly, we conclude that    $(g(q_n))_\nnn$ converges.
		Setting $y=\overline{z}$ in \cref{e:limits} yields
		\begin{equation}
			\label{e:sum:fg:lim}
			f(p_n)+g(q_n)
			\to f(\overline{z})+g(\overline{z}-v)\ge \mu.
		\end{equation}
		Choosing $y$ so that
		$ f(y)+g(y-v)$ is as close to $\mu $
		as desired, we see that \cref{e:limits} yields
		\begin{equation}
			f(p_n)+g(q_n)
			\to f(\overline{z})+g(\overline{z}-v)=\mu.
		\end{equation}
		Therefore, we conclude that
		$\overline{z}$ is a minimizer of
		$f+g(\cdot-v)$.
    
		\cref{prop:clusters:ii}:
		Combine \cref{prop:clusters:i}
		and \cref{e:seq:strong:lim}.
		
		\cref{prop:clusters:iv}\&\cref{prop:clusters:v}:
		Observe that the lower semicontinuity
		of $f$ and $g$ respectively implies
		\begin{subequations}
			\begin{align}
				f(\overline{z})   & \le \lim f(p_n)
				\\
				g(\overline{z}-v) & \le \lim g(q_n).
				\label{e:sum:fg:lim:g}
			\end{align}
		\end{subequations}
		Suppose for eventual contradiction
		that $f(c)< \lim f(p_n)$.
		Then \cref{e:sum:fg:lim} implies that
		$\lim g(q_n)>g(\overline{z}-v)$
		which, by \cref{e:sum:fg:lim:g}, 
		is absurd.
    
    \cref{prop:clusters:iii}: 
    Combine \cref{prop:clusters:iv}, \cref{prop:clusters:v}, and
    \cref{prop:clusters:i}.
	\end{proof}
	
	\begin{proposition}
		\label{prop:clusters:conv}
		Let $x\in X$. Then the following hold:
		\begin{enumerate}
    \setlength\itemsep{0.35em}
			\item
			\label{prop:clusters:conv:i}
			The sequence $(\prox_fT^n x)_\nnn$
			converges weakly to a minimizer of
			$f+g(\cdot-v)$. 
			\item
			\label{prop:clusters:conv:ii}
			The sequence $(\prox_g\refl_fT^n x)_\nnn$ converges weakly to a minimizer of
			$f(\cdot+v)+g$.
		\end{enumerate}
	\end{proposition}
	
	\begin{proof}
		\cref{prop:clusters:conv:i}:
		Recalling \cref{e:v=vd}, it follows from
		\cref{prop:pd:sh:bdd:fejer}\cref{prop:pd:sh:bdd:fejer:x}
		applied with $(A,B)$ replaced by $(\partial f,\partial g(\cdot-v))$ that
		the sequence $(\prox_fT^n x,-v+\prox_{f^*}T^nx+nv )_\nnn$
		is Fej\'er monotone with respect to
		$Z\times K$.
		Now let $z_1$ and $z_2$ be two weak cluster points of
		$(\prox_fT^n x)_\nnn$.
		On the one hand, \cref{prop:clusters}\cref{prop:clusters:i}
		and \cref{prop:fg:zv}\cref{prop:fg:zv:ii}
		imply that
		\begin{equation}
			\label{e:zv:ortho}
			\{z_1,z_2\}\subseteq\argmin_{x\in X}(f+g(\cdot-v))={Z};
			\quad \text{hence},\quad z_1-z_2 \in {Z}-{Z}
			.
		\end{equation}
		On the other hand,
		\cite[Lemma~2.2)]{95}
		implies that  $z_1-z_2\in ({Z}- {Z})^{\perp}$.
		Combining with \cref{e:zv:ortho} we
		conclude that ${z_1-z_2}\in ({Z}-{Z})
		\cap({Z}-{Z})^{\perp}=\{0\}$.
		Hence, $z_1=z_2$.
		
		\cref{prop:clusters:conv:ii}:
		Combine \cref{prop:clusters:conv:i} and \cref{e:seq:strong:lim}.
	\end{proof}
	
	\begin{proposition}
		\label{prop:constrained}
    Suppose that $f=\iota_C$, where 
		$C$ is a nonempty closed convex subset of $X$.
		Let $x\in X$.
		Then there exists $\overline{z}\in Z\subseteq C$
		such that the following hold:
		\begin{enumerate}
    \setlength\itemsep{0.35em}
			\item
			\label{prop:constrained:i}
			$\Pj{C} T^n x\weakly \overline{z}$
			and $ \overline{z}$  is a minimizer of
			$g(\cdot-v)$ over $C$.
			\item
			\label{prop:constrained:ii}
			$\Pj{Z} \Pj{C }T^n x\to \overline{z}$.
		\end{enumerate}
	\end{proposition}
	\begin{proof}
		\cref{prop:constrained:i}:
		Clearly, $\prox_f=\Pj{C}$.
		Now combine with \cref{prop:clusters:conv}\cref{prop:clusters:conv:i}.
		
		\cref{prop:constrained:ii}:
		Applying
		\cref{prop:pd:sh:bdd:fejer}\cref{prop:pd:sh:bdd:fejer:x}
		with $(A,B)$ replaced by
		$(\N{C},\partial g)$ we learn that the
		sequence $((\Pj{C}T^n x, (\Id-\Pj{C})T^n x+(n-1)v))_\nnn$
		is Fej\'er monotone with respect to $Z\times K$.
		Combining this with
		\cite[Proposition~5.7]{BC}
		we learn that  $(\exists (z,k)\in Z\times K)$
		such that
		$(\Pj{Z}\Pj{C}T^n x,\Pj{K}((\Id-\Pj{C})T^n x+(n-1)v))\to (z,k)$.
		In particular,
		$\Pj{Z}\Pj{C}T^n x\to z$.
		On the other hand, \cite[Corollary~5.8]{74} yields
		$\Pj{Z}\Pj{C}T^n x\weakly \overline{z}$.
	\end{proof}
	
	\begin{example}
		\label{ex:main:conv}
		Suppose that $X=\RR^m$, where 
		$m\ge 1$.
		Let $b\in \RR^m$,
		let $u\in \RR^m\smallsetminus\{0\}$,
		let $\eta\in \RR$,
		set $B\coloneqq\menge{x\in \RR^m}{\scal{x}{u}\le \eta}$,
		and
		set $C\coloneqq\menge{x\in \RR^m}{-c\le x\le c}$,
		where $c\in[0,+\infty]^m$.
		Suppose that  $f=\iota_{B}$
		and that 
		\begin{equation}
			g\colon \RR^m \to \left]-\infty
      ,+\infty\right]
			\colon x\mapsto
			\begin{cases}
				\sum_{i=1}^{m}\abs{x_i}, & \text{if } x\in C;
				\\
				+\infty,                 & \text{otherwise}.
			\end{cases}
		\end{equation}
		Let $x\in X$.
		Then the following hold:
		\begin{enumerate}
			\setlength\itemsep{0.35em}
			\item
			\label{ex:main:conv:0}
			$\partial f=\N{B}$.
			\item
			\label{ex:main:conv:i}
			$g=\norm{\cdot}_{1}+\iota_C$.
			\item
			\label{ex:main:conv:ii}
			$\partial g=\partial \norm{\cdot}_{1}+N_C$.
			\item
			\label{ex:main:conv:iii}
			$\dom f-\dom g=\dom \partial f-\dom \partial g=B-C=\overline{B-C}$.
			\item
			\label{ex:main:conv:iv}
			$\dom f^*+\dom g^*=\RR^m$.
			\item
			\label{ex:main:conv:v}
			$\cran (\Id-T)=B-C$.
			\item
			\label{ex:main:conv:vi}
			$v_D=\Pj{{B-C}}(0)\in B-C$.
			\item
			\label{ex:main:conv:vi:0}
			$v_D=0\siff B\cap C\neq \fady$.
			\item
			\label{ex:main:conv:vii}
			$v_R=0$.
			\item
			\label{ex:main:conv:viii}
			$v=v_D$.
			\item
			\label{ex:main:conv:ix}
			$Z
			=\displaystyle\argmin_{y\in \RR^m}\big(\iota_B(y)+g(y-v)\big)\neq \fady$.
			\item
			\label{ex:main:conv:x}
			$\prox_fx=\Pj{B}  x
			=\begin{cases}
				x,                                   & \text{if } \scal{x}{u}\le \eta;
				\\
				x+(\eta-\scal{x}{\eta})u/\norm{u}^2, & \text{if }\scal{x}{u}>\eta .
			\end{cases}$
			\item
			\label{ex:main:conv:xi}
			$\prox_g x=(\xi_i)_{i=1}^{n}$,
			where
			$\xi_i=\min\{\max\{\abs{x_i}-1,0\},c_i\}\sign(x_i)$.
			\item
			\label{ex:main:conv:xii}
			$T=\Id-\Pj{B}+\prox_g(2\Pj{B}-\Id)$.
			\item
			\label{ex:main:conv:xiii}
			$\Pj{B}T^n x\to \lim \Pj{Z}\Pj{B}T^n x\in Z$.
		\end{enumerate}
	\end{example}
	\begin{proof}
		\cref{ex:main:conv:0}\&\cref{ex:main:conv:i}:
		This is clear.
		
		\cref{ex:main:conv:ii}:
		This follows from combining
		\cite[Corollary~16.48(iii)~and~Example~16.13]{BC}.
		
		\cref{ex:main:conv:iii}:
		Clearly, $\dom \partial f=\dom \N{B}=\dom \iota_B=B$.
		Moreover, $\dom g=\dom \norm{\cdot}_1\cap \dom \iota_C
		=\dom \partial \norm{\cdot}_1\cap\dom \N{C}=C
		$.  Finally, observe that $B$ is compact and $C$
		is  closed;
		therefore, $B-C$ is closed.
		
		\cref{ex:main:conv:iv}:
		Indeed,
		it follows from, e.g., \cite[Corollary~21.25]{BC},
		and \cite[Remark~on~page~216]{Rock1970}
		that
		\begin{equation}
			\label{eq:ran:X}
			\RR^m= \ran \partial g=\dom \partial g^*
			\subseteq \dom g^*\subseteq \RR^m .
		\end{equation}
		Hence, $\dom g^*= \RR^m$
		and the conclusion follows.
		
		\cref{ex:main:conv:v}:
		Combine \cref{ex:main:conv:iii}, \cref{ex:main:conv:iv},
		and \cref{rem:assump:free:fd}.
		
		\cref{ex:main:conv:vi}\&\cref{ex:main:conv:vi:0}:
		Combine \cref{ex:main:conv:iii}
		and \cref{e:def:vD:vR}.
		
		\cref{ex:main:conv:vii}:
		This is a direct consequence
		of \cref{e:def:vD:vR} and \cref{ex:main:conv:iv}.
		
		\cref{ex:main:conv:viii}:
		Combine \cref{ex:main:conv:vii}
		and \cref{prop:v:decom}\cref{prop:v:decom:viii}.
		
		\cref{ex:main:conv:ix}:
		The first identity follows from
		combining  \cref{ex:main:conv:0},
		\cref{ex:main:conv:ii}
		and \cref{prop:fg:zv}\cref{prop:fg:zv:i}.
		Next, observe that
		$f$ and $g$ are polyhedral functions.
		Moreover,
		\cref{ex:main:conv:iii}
		and
		\cref{ex:main:conv:vi}
		imply that
		$\dom f\cap (v+\dom g)=B\cap(v+C)\neq \fady$.
		Therefore, \cite[Theorem~23.18]{Rock70} yields that
		$\partial(f+g(\cdot-v))=\partial f+\partial g(\cdot-v)$.
		We learn that $\partial f+\partial g(\cdot-v)$
		is maximally monotone and, in view of
		\cite[Theorem~3.13]{BMW},
		that $\ri \ran (\partial f+\partial g(\cdot-v))=\ri \ran \partial f+\ri \ran \partial g
		\subseteq  \ran \partial f+ \ran \partial g$.
		Observe that \cref{eq:ran:X} implies that
		$\ran \partial g(\cdot-v)=\ran \partial g=\RR^m=\ri \RR^m=\ri \ran \partial g(\cdot-v)$.
		Altogether, we learn that $\ran (\partial f+\partial g(\cdot-v))=\RR^m$,
		hence $Z=\zer (\partial f+\partial g(\cdot-v))\neq \fady$.
		Now combine with
		\cref{prop:fg:zv}\cref{prop:fg:zv:ii}.
		
		\cref{ex:main:conv:x}:
		This follows from, e.g.,
		\cite[Proposition~29.20(iii)]{BC}.
		
		\cref{ex:main:conv:xi}:
		This follows from
		\cite[Example~6.23]{Beck2017}.
		
		\cref{ex:main:conv:xii}:
		This is \cref{e:def:T} applied with $(A,B)$
		replaced by $(\partial f, \partial g)$.
		
		\cref{ex:main:conv:xiii}:
		Apply \cref{prop:constrained} with $(C,g)$ replaced by $(B,\norm{\cdot}_1+\iota_C) $
	\end{proof}

	Having collected already all pieces required for its proof,
	we now summarize our work in the following: 
	
	\begin{theorem}{\bf (main result: primal shadows converge!)}
		\label{thm:conv}
    Let $f,g$ be in $\Gamma_0(X)$. 
		Suppose that $0\in { \dom  f^*+\dom  g^*}$,
		that $v\in \ran (\Id-T)$, and that $\zer(\partial f+\partial g(\cdot-v) )\neq \fady$.
    Set $\mu\coloneqq \min_{x\in X}(f(x+v)-g(x))$. 
		Let $x\in X$. 
		Then there exists a vector $\overline{z}\in  X$
		such that  the following hold:
		\begin{enumerate}
			\setlength\itemsep{0.35em}
			\item
			\label{thm:conv:i}
			$\prox_fT^n x\weakly  \overline{z}$
			and $ \overline{z}$ is a minimizer of
			$f+g(\cdot-v)$.
			\item
			\label{thm:conv:ii}
			$\prox_g\refl_fT^n x\weakly  \overline{z}-v$
			and $\overline{z}-v$ is a minimizer of
			$f(\cdot+v)+g$.
			\item
			\label{thm:conv::iii}
			$f(\prox_fT^n x)+g(\prox_g\refl_fT^n x)\to f(\overline{z})+g(\overline{z}-v)
			=\mu$. 
			\item
			\label{thm:conv::iv}
			$f(\prox_fT^n x)\to f(\overline{z})$.
			\item
			\label{thm:conv::v}
			$g(\prox_g\refl_fT^n x)\to g(\overline{z}-v)$.
		\end{enumerate}
	\end{theorem}
	\begin{proof}
		Note that $0\in\dom f^*+\dom g^* \subseteq \overline{\cdom f^*+\cdom g^*}
		= \overline{\cran\partial f+\cran\partial g} = \overline{\ran\partial f
			+\ran \partial g} = \overline{R}$ $\Rightarrow$ $v_R=0$. 
		Now combine 
		with \cref{prop:fg:zv},
		\cref{prop:clusters:conv},
		and \cref{prop:clusters}\cref{prop:clusters:iii}--\cref{prop:clusters:v}. 
	\end{proof} 
	
	Finally, let us ``dualize'' \cref{thm:conv} by applying 
	it to $(f^*,g^{*\veet})$ to obtain that the dual shadows
	solve the normal version of the Fenchel dual problem. 
	
	\begin{corollary}{\bf (dual shadows converge!)}
		\label{thm:conv:dual}
		Let $f,g$ be in $\Gamma_0(X)$. 
		Suppose that $0\in { \dom  f-\dom  g}$,
		that $v\in \ran (\Id-T)$, and that $\zer(\partial f^*
		+\partial (g^{*\veet})(\cdot-v) )\neq \fady$.
		Let $x^*\in X$. 
		Then there exists a vector $\overline{k}\in  X$
		such that  the following hold:
		\begin{enumerate}
			\setlength\itemsep{0.35em}
			\item
			\label{thm:conv:dual:i}
			$\prox_{f^*}T^n x\weakly  \overline{k}$
			and $ \overline{k}$ is a minimizer of
			$f^*+g^{*\veet}(\cdot-v)$.
			\item
			\label{thm:conv:dual:ii}
			$\prox_{g^{*\veet}}\refl_{f^*}T^n x\weakly  \overline{k}-v$
			and $\overline{k}-v$ is a minimizer of
			$f^*(\cdot+v)+g^{*\veet}$.
			\item
			\label{thm:conv:dual:iii}
			$f^*(\prox_{f^*}T^n x)
			+g^{*\veet}(\prox_{g^{*\veet}}\refl_{f^*}T^n x)
			\to 
			\displaystyle\min_{x^*\in X} (f^*(x^*)+g^{*\veet}(x^*-v))$.
			\item
			\label{thm:conv:dual:iv}
			$f^{*}(\prox_{f^*}T^n x)\to f^*(\overline{k})$.
			\item
			\label{thm:conv:dual:v}
			$g^*(-\prox_{g^{*\veet}}\refl_{f^*}T^n x)\to g^*(-\overline{k}+v)$.
		\end{enumerate}
	\end{corollary}
	\begin{proof}
		Observe that, likewise $f$ and $g$,
		both $f^*$ and $g^{*\veet}$ are convex, lower semicontinuous 
		and proper. Moreover, $(f^*)^*=f, (g^{*\veet})^*=g^\veet$.
		Hence 
		\begin{equation}
			\label{e:dom:conj}
			\dom (f^*)^*+\dom (g^{*\veet})^*=\dom f-\dom g.
		\end{equation}
		Finally, observe that 
		$(\partial f^*,\partial (g^{*\veet}))
		=((\partial f)^{-1}, (\partial g)^{-\oveet})$.
		Consequently, \cref{e:def:T:fg}
		and \cite[Lemma~3.6~on~page~133]{EckThesis}
		implies that 
		\begin{equation}
			\label{e:T:sd}
			T_{(\partial f^*,\partial g^{*\veet})}=T.
		\end{equation}
		\cref{thm:conv:dual:i}--\cref{thm:conv:dual:v}:
		Combine \cref{e:dom:conj},
		\cref{e:T:sd}
		and \cref{thm:conv} 
		applied with $(f,g)$ replaced by $(f^*, g^{*\veet})$.
	\end{proof}

	\section*{Acknowledgments}
	
	The research of HHB and WMM was partially supported by Discovery Grants
	of the Natural Sciences and Engineering Research Council of
	Canada.

\end{document}